\documentclass[11pt]{amsart}
\usepackage{graphicx,amsfonts,amssymb,amsmath,amsthm,url,
  verbatim,amscd, fullpage}
  \usepackage{color}
\usepackage[usenames,dvipsnames]{xcolor}
\usepackage[normalem]{ulem}
\usepackage{pdfsync}
\usepackage{booktabs}
\usepackage[hyperfootnotes=false, colorlinks, citecolor=RoyalBlue,
urlcolor=blue, linkcolor=blue ]{hyperref}

\let\origmaketitle\maketitle
\def\maketitle{
  \begingroup
  \def\uppercasenonmath##1{} % this disables uppercasing title
  \let\MakeUppercase\relax % this disables uppercasing authors
  \origmaketitle
  \endgroup
  }

\newtheorem{lemma}{Lemma}[section]
\newtheorem{theorem}{Theorem}

\newtheorem{corollary}[lemma]{Corollary}
\newtheorem{proposition}[lemma]{Proposition}
\newtheorem{definition}[lemma]{Definition}

\newtheorem{remark}[lemma]{Remark}

\allowdisplaybreaks

\renewcommand{\H}{\mathbb H}
\newcommand{\A}{{\mathbb A}}
\newcommand{\m}{{\rm max}}
\newcommand{\cond}{{\rm cond}}

\newcommand{\f}{{\mathbf f}}
\renewcommand{\c}{{\mathbf c}}

\newcommand{\J}{{\mathcal J}}
\newcommand{\Q}{{\mathbb Q}}
\newcommand{\Z}{{\mathbb Z}}

\renewcommand{\O}{{\mathcal O}}
\renewcommand{\L}{{\mathcal L}}

\newcommand{\R}{{\mathbb R}}
\newcommand{\C}{{\mathbb C}}
\newcommand{\bs}{\backslash}

\renewcommand{\S}{\mathfrak S}
\renewcommand{\P}{\mathcal P}

\newcommand{\GL}{{\rm GL}}

\newcommand{\SL}{{\rm SL}}
\newcommand{\SO}{{\rm SO}}

\newcommand{\sgn}{{\rm sgn}}

\newcommand{\ur}{{\bf ur}}

\newcommand{\trace}{{\rm tr}}

\newcommand{\tr}{{\rm Tr}}
\newcommand{\nr}{{\rm nr}}

\newcommand{\mat}[4]{{\setlength{\arraycolsep}{0.5mm}\left(
      \begin{array}{cc}#1&#2\\#3&#4\end{array}\right)}}

\def\SL{\operatorname{SL}}
\def\vol{\operatorname{vol}}
\def\GL{\operatorname{GL}}

\renewcommand{\Im}{\mathrm{Im}}
\def\eps{\epsilon}

\begin{document}

\bibliographystyle{plain}

\title[{Sup-norm in the level aspect for compact arithmetic surfaces}]{\Large{Sup-norms of eigenfunctions in the level aspect for compact arithmetic surfaces}}

\author[Abhishek Saha]{\large{Abhishek Saha}}
\address{School of Mathematical Sciences\\
  Queen Mary University of London\\
  London E1 4NS\\
  UK}

\begin{abstract} Let $D$ be an indefinite quaternion division algebra over $\Q$. We approach the problem of bounding the sup-norms of automorphic forms $\phi$ on $D^\times(\A)$ that belong to irreducible automorphic representations and transform via characters of unit groups of orders of $D$. We obtain a non-trivial upper bound for $\|\phi\|_\infty$ in the level aspect that is valid for \emph{arbitrary} orders. This generalizes and strengthens previously known upper bounds for $\|\phi\|_\infty$  in the setting of newforms for Eichler orders.  In the special case when the index of the order in a maximal order is a \emph{squarefull} integer $N$, our result specializes to $\|\phi\|_\infty \ll_{\pi_\infty, \eps} N^{1/3 + \eps} \|\phi\|_2$.

A key application of our result is to automorphic forms $\phi$ which correspond at the ramified primes to either minimal vectors (in the sense of \cite{HNS}), or $p$-adic microlocal lifts  (in the sense of \cite{nelson-padic-que}). For such forms, our bound specializes to  $\| \phi\|_\infty \ll_{ \eps} C^{\frac16 + \eps}\|\phi\|_2$ where $C$ is the  conductor of the representation $\pi$ generated by  $\phi$. This improves upon the previously known \emph{local bound} $\|\phi\|_\infty \ll_{\lambda, \eps} C^{\frac14 + \eps}\|\phi\|_2$ in these cases. %In the case of newforms $f$ of square conductor $C$ whose  character $\chi$ has conductor $M$, our result implies that  $\| f \|_\infty \ll_{\lambda, \eps} C^{\frac13 + \eps}\|f\|_2$, which improves upon the local bound  \cite{marsh15,saha-sup-level-hybrid} whenever $C^2|M^3$.

%We compare our result with the \emph{local bound} for automorphic forms $\phi$ inside automorphic representations $\pi$ of $D^\times$ and observe that our result is stronger for several families. The local bound (in the sense of \cite{marsh15,saha-sup-level-hybrid}) says that $\| \phi\|_\infty \ll_{\eps}  (C_1M_1)^{\frac12 + \eps}\|\phi\|_2$ where $C_1$ is the smallest integer such that $\mathrm{cond}(\pi)|C_1^2$ and $M_1 := \frac{\mathrm{cond}(\omega_\pi)}{\gcd(C_1, \mathrm{cond}(\omega_\pi))}.$  Our result implies the bound $\| \phi\|_\infty \ll_{\eps} (C_1M_1)^{\frac13 + \eps}\|\phi\|_2$ for  forms $\phi$ that are either (i) of minimal type (in the sense of \cite{HNS}), or (ii) $p$-adic microlocal lifts (in the sense of \cite{nelson-padic-que}), or (iii) newforms of  maximally ramified central character and squarefull conductor.

\end{abstract}

\maketitle

\section{Introduction}\label{s:introduction}
\subsection{Background}Let $\phi=\otimes_v \phi_v$ be a  cuspidal automorphic form on $D^\times(\A)$ where $D$ is an indefinite quaternion  algebra over $\Q$. The sup-norm problem, which has seen a lot of interest recently, is concerned with bounding  $\frac{\|\phi\|_\infty}{\| \phi \|_2}$ in terms of natural parameters of $\phi$. When the primary focus is the dependance of the bound on parameters related to the ramified primes or to the underlying level structures associated to $\phi$ (with the dependance on the archimedean parameters suppressed), this is known as the level-aspect sup-norm problem.

In the split case where $D^\times = \GL_2$, this problem has been heavily studied in the special case that $\phi$ is a global \emph{newform} \cite{blomer-holowinsky, templier-sup, harcos-templier-1, harcos-templier-2, templier-sup-2, sahasuplevel, saha-sup-level-hybrid}. In this case,  $\phi$ transforms by a character of the subgroup\footnote{Technically, we need to consider the adelic counterpart of this subgroup.} $\Gamma_0(N)$ of norm 1 units in the standard Eichler order of level $N$, where $N$  equals the  conductor of the representation $\pi$ generated by $\phi$. The best upper bound currently known in the case of newforms on $\GL_2$  is due to the present author \cite{saha-sup-level-hybrid} and in the \emph{trivial  character} case this bound reads\footnote{As usual, we use the notation
$A \ll_{x,y,\ldots} B$
to signify that there exists
a positive constant $C$, depending at most upon $x,y,\ldots$,
so that
$|A| \leq C |B|$.}  \begin{equation}\label{saha:result}\|\phi\|_\infty
  \ll_{\pi_\infty,\eps} N_0^{1/6 + \eps}
  N_1^{1/3+\eps}\|\phi \|_2,\end{equation} for any $\epsilon>0$, where we write $N=N_0N_1$ with
$N_0$ the largest integer such that $N_0^2$ divides $N$. More recently, the sup-norm problem has also been considered for newforms on $\GL_2$ over number fields, we refer the reader to \cite{blomer-harcos-milicevic-maga, assing} for this.

In the compact situation where $D$ is a \emph{division} algebra, less work has been done. As in the split case, the \emph{trivial bound} in the level aspect is \begin{equation} \label{e:trivial}\|\phi\|_\infty \ll_{\pi_\infty, \eps} N^{1/2 + \eps}\|\phi \|_2,
\end{equation} for any automorphic form $\phi$ that transforms by a  character of the unit group of an order of level $N$ (see below for a more precise definition of this). The first improvement in this setting was due to Templier \cite{ templier-sup-errata,templier-sup}, who proved that for $\phi$ a global newform with respect to an Eichler order of level $N$, one has the bound \begin{equation}\label{e:templier}\|\phi\|_\infty \ll_{D, \pi_\infty, \eps} N^{11/24 + \eps}\|\phi \|_2.\end{equation} More recently, Marshall \cite{marsh15} proved the bound (again, only in the setting of newforms for Eichler orders of level $N$): \begin{equation}\label{e:local}\|\phi\|_\infty \ll_{D, \pi_\infty} \bigg(N_1\prod_{p|N} (1 + p^{-1})\bigg)^{1/2}\|\phi \|_2,\end{equation} with $N_1$ as in \eqref{saha:result}. (It was noted in \cite{marsh15} that the same bound also holds in the split case of $D^\times = \GL_2$, provided  one restricts the domain to a fixed compact set.) Note that Marshall's bound is better than Templier's when $N$ is squarefull, but worse when $N$ is squarefree. This reflects the fact that Marshall's bound is the local bound, which essentially coincides with the trivial bound for newforms of squarefree level, but is stronger than the trivial bound in general. We discuss this distinction in more detail in Section \ref{s:localbdintro}; see also Remark \ref{rem:local} for a more precise formulation.

\subsection{The main result} For the rest of this paper, let $D$ be a fixed indefinite quaternion division algebra over $\Q$.  We now describe (a simplified version of) the main result of this paper, which deals with the compact case and for the first time, improves upon the trivial bound \eqref{e:trivial} for completely arbitrary orders. Let $\O^\m$ denote a maximal order  of $D$.  %For any integer $M$, let $\O^\m(M)$ be the ``principal congruence order" defined by $$\O^\m(M) := \Z + M\O^\m.$$
For any order  $\O \subseteq \O^\m$  of $D$ define the level $$N_\O = [\O^\m: \O].$$ %Given an order $\O$, there exists an integer $M$ such  that $\O^\m(M) \subseteq \O$ (in fact  any multiple $M$ of $N_\O$ has this property).
Given an order $\O \subseteq \O^\m$,  define the (adelic) congruence subgroup $K_\O$ of $D^\times (\A_\f)$ (where $\A_\f$ denotes the finite adeles) by $$K_\O = \prod_p \O_p^\times$$ where the product is taken over all primes, and where we denote $\O_p = \O \otimes_{\Z} \Z_p$.

  Let $\pi$ be an irreducible automorphic representation of $D^\times(\A)$ with unitary central character. Since we are in the division algebra case, $\pi$ is unitary and cuspidal. By the multiplicity one theorem, we can uniquely identify $V_\pi$ with a certain space of automorphic forms on $D^\times(\A)$. For $\phi \in V_\pi$, we define as usual $$\| \phi \|_2 = \int_{ D^\times \bs D^\times(\A)/\A^\times} |\phi(g)|^2 dg.$$ Let $\O \subseteq \O^\m$ be an order and $\chi$ be a character of $K_\O$.  Note that we may write $\chi = \prod_p \chi_p$ where $p$ traverses the primes, and $\chi_p$ is a character of $\O_p^\times$ with $\chi_p$ trivial for almost all $p$. The compactness of $K_\O$ and the continuity\footnote{All our characters are assumed to be continuous.} of $\chi$ automatically implies that $\chi$ is of finite order. We define $V_\pi(K_\O, \chi) \subseteq V_\pi$ to be the subspace of functions $\phi \in V_\pi$ that have the transformation property \begin{equation}\label{e:transf}\phi(gk) = \chi(k) \phi(g) \text{ for all } k\in K_\O, \ g \in D^\times(\A).\end{equation} Given any non-zero $\phi \in V_\pi(K_\O, \chi) $, we wish to bound the sup-norm $\frac{\|\phi\|_\infty}{\|\phi \|_2}$ in terms of the level $N_\O$.

  As is clear from the earlier discussion, previous work on this topic has been focussed on the case where $\phi$ is a newform and $\O$ is an Eichler order with level $N_\O$ equal to the conductor of $\pi$. This restriction to newforms and Eichler orders is quite limiting as it does not capture the behavior of several natural families of automorphic forms. For example, there is an emerging theory of automorphic forms of minimal type  \cite{HNS,Hu:minimal,HN-minimal,hu-shu-yin}; such forms transform naturally with respect to characters of unit groups of certain \emph{non-Eichler} Bass orders. The aim of this paper is to prove for the first time a non-trivial upper bound for the sup-norm in the case of general orders.
  \medskip

\textbf{Theorem A.} (see Theorem~\ref{t:maingen}) \emph{Let $\O$ be an order of $D$ and denote $N=[\O^\m: \O]$. Let $N_1$ be as in \eqref{saha:result}. Let $\phi \in V_\pi(K_\O, \chi) $ where $\pi$ is an irreducible  automorphic representation of $D^\times(\A)$  with unitary central character, and $\chi$ is a  character of $K_\O.$ Suppose that $\phi_\infty$ corresponds to the vector of lowest non-negative weight\footnote{This assumption on $\phi_\infty$ is merely for convenience; our result can be stated without this assumption but then the implied constant will depend on $\phi_\infty$.} in $\pi_\infty$. Then we have $$\|\phi\|_\infty \ll_{D, \pi_\infty, \eps} N^{\eps} \min(\max(N^{1/3}, N_1^{1/2}), N^{11/24})\|\phi \|_2.$$}

For squarefree $N$, our Theorem implies that $\|\phi\|_\infty \ll_{D, \pi_\infty, \eps} N^{11/24 + \eps}\|\phi \|_2$. However, when $N$ is squarefull (i.e., every prime that divides $N$ does so with exponent at least 2) Theorem A gives a much stronger ``Weyl-type" exponent.

\medskip

\textbf{Corollary.} \emph{Let the notations and assumptions be as in Theorem A and  assume that $N$ is \emph{squarefull}. Then we have $$\|\phi\|_\infty \ll_{D, \pi_\infty, \eps} N^{\frac13+ \eps}\|\phi \|_2.$$}

For an explanation why we get such different exponents for squarefree and squarefull levels, see Section \ref{s:introkeyideas} of this introduction. An interesting fact is that this squarefree/squarefull dichotomy in the sup-norm bound is also present in the case of newforms on $\GL_2$ (see Section 1.3 of \cite{saha-sup-level-hybrid}), but for utterly different reasons!

\begin{remark}Our main result (Theorem \ref{t:maingen}) is significantly more general than Theorem A above, because it does not require that the space generated by $\phi$ under the action of $K_\O$ is one-dimensional.
\end{remark}
\subsection{A classical reformulation}

In this subsection, we reformulate Theorem A in the language of Hecke-Laplace eigenfunctions on the upper-half plane, which may be helpful for those who are more familiar with the classical language.

We let $d$ denote the reduced discriminant of $D$ and we fix an isomorphism $\iota_\infty: D \otimes_\Q \R \simeq M_2(\R)$. This leads to an embedding $D \hookrightarrow M_2(\R)$ which we also denote by $\iota_\infty$.

 For any order $\O$ of $D$,  associate a discrete subgroup $\Gamma_\O$ of $\SL_2(\R)$ as follows:
$$\Gamma_\O = \{ \gamma \in \iota_\infty(\O), \ \det(\gamma)=1\}.$$ Note that $\Gamma_\O\bs \H$ is a compact hyperbolic surface.

Now, let $\chi = \prod_p \chi_p$ be a unitary character of $K_\O$ as in the previous subsection.  We have the identity \begin{equation}\label{e:char}\iota_\infty^{-1}(\Gamma_\O) = D^\times(\Q) \cap D^\times(\R)^+ K_\O,\end{equation} where $D^\times(\R)^+ $ consists of the elements of $D^\times(\R)\simeq \GL_2(\R)$ with positive reduced norm (positive determinant). We extend $\chi$ to $D^\times(\R)^+ K_\O$ by making it trivial on $D^\times(\R)^+$. Hence \eqref{e:char} allows us to define a character on $\Gamma_\O$ which (by abusing notation) we also denote by $\chi$.
Let  $N_\chi$ be an integer such that \begin{equation}\label{d:congchar}\O^\m(N_\chi):= \Z + N_\chi \O^\m\subseteq \O,  \qquad \chi_p|_{1+N_\chi \O^\m_p} = 1 \text{ for all primes }p.\end{equation} (The existence of $N_\chi$ is clear.)  The character $\chi: \Gamma_\O \rightarrow S^1$ is trivial on the principal congruence subgroup $\Gamma_{\O^\m(N_\chi)}$, which is a normal subgroup of $\Gamma_\O$.  In particular, $\chi$ is a \emph{congruence character}, i.e., it is trivial on a principal congruence subgroup.

We let $C^\infty(\Gamma_\O \bs \H, \chi^{-1})$ denote the space of smooth functions $f: \H \rightarrow \C$ such that $$f(\gamma z) = \chi^{-1}(\gamma) f(z)$$ for all $\gamma \in \Gamma_\O$. For each $f \in C^\infty(\Gamma_\O \bs \H, \chi^{-1})$, we define $$\|f \|_2 = \left(\vol(\Gamma_\O\bs \H)^{-1}\int_{\Gamma_\O\bs \H}|f(z)|^2 dz\right)^{1/2}$$ where $dz = c \frac{dx dy}{y^2}$ is any $\SL_2(\R)$-invariant measure on $\H$. On  $C^\infty(\Gamma_\O \bs \H, \chi^{-1})$ there exist \emph{Hecke operators} $T_n$ for each positive integer $n$ defined as follows: $$(T_n f) (z) = \sum_{\gamma \in \Gamma_{\O^\m(N_\chi)}  \bs S_{\O^\m(N_\chi)}(n) } f(\gamma z),$$     where $N_\chi$ is chosen as above, and the subset $S_{\O^\m(N_\chi)}(n)$ of $\GL_2(\Q)$ is defined by $$S_{\O^\m(N_\chi)}(n) = \{ \gamma \in \iota_\infty(\O^\m(N_\chi)), \ \det(\gamma)=n\}.$$ It can be checked that the definition of $T_n$ given above is  \emph{independent} of all choices, including the choice of $N_\chi$, and is well-defined on the space $C^\infty(\Gamma_\O \bs \H, \chi^{-1})$. For all $(n, dN_\chi) =1$, these operators are normal and commuting.

Now, let $\phi$ and $\pi$ be as in Theorem A. Assume that $\phi$ is right-invariant by $\iota_\infty^{-1}(\SO(2))$. Define a function $f_\phi:\H \mapsto \C$ via the equation $$f_\phi(z) = \phi(g_\infty)$$ where $g_\infty \in D^\times(\R)^+$ is any matrix such that $\iota_\infty(g_\infty) i = z$. Then $f_\phi$
  has the following properties:
      \begin{enumerate}

     \item $f_\phi \in C^\infty(\Gamma_\O \bs \H, \chi^{-1})$.

     \item $f_\phi$ satisfies  $(\Delta + \lambda) f_\phi =
0$ where $\Delta := y^{-2}
(\partial_x^2 + \partial_y^2)$.

 \item $f_\phi$ is a simultaneous eigenfunction of the Hecke operators $T_n$ for all positive integers $n$ with $(n,dN_\chi)=1$.
      \end{enumerate}

      In other words, $f_\phi$ is a \emph{Maass form} for $\Gamma_\O$ with character $\chi^{-1}$ and Laplace eigenvalue $\lambda$ that is a Hecke eigenform at the good primes.
  \emph{Theorem A can be reformulated as an upper-bound on the sup-norms of such $f_\phi$: $$\|f_\phi\|_\infty \ll_{D, \lambda, \eps} N^{\eps} \min(\max(N^{1/3}, N_1^{1/2}), N^{11/24})\|f_\phi \|_2.$$}

This follows from the fact that $\sup_{g\in D^\times(\A)}|\phi(g)| = \sup_{z \in \H} |f_\phi(z)|.$

\begin{remark} In fact, every Maass form $f$ for $\Gamma_\O$ with character $\chi^{-1}$ and Laplace eigenvalue $\lambda$ that is a Hecke eigenform at the good primes, can be obtained in the above way, i.e., $f=f_\phi$ for some $\phi$ as in Theorem A. This can be proved using the technique of adelization. We omit the details of this in the interest of brevity.
\end{remark}

\subsection{The local bound and application to minimal vectors}\label{s:localbdintro}
\emph{For this subsection, we assume for simplicity that $\pi$ has trivial central character}. We compare Theorem A with the \emph{local bound} in the level aspect for automorphic forms $\phi$ inside automorphic representations $\pi$ of $D^\times(\A)$.  By the local bound for $\phi$, we mean the immediate bound emerging from the adelic pre-trace formula where the local test function at each ramified prime is chosen to be the restriction (to a maximal compact subgroup) of the \emph{matrix coefficient} of $\phi_p$. In fact, an explicit computation performed in \cite{marsh15,saha-sup-level-hybrid} for the case of newforms together with the principle of formal degrees, allows us to write down this bound in terms of the conductor of $\pi$.

More precisely,  let $\pi$ be as in Theorem A such that $\pi$ has trivial central character and $\pi_p$ is one-dimensional at each prime dividing $d$. Then, for all $\phi \in V_\pi$ that satisfy a mild condition\footnote{The condition is that for some $g \in D^\times(\A_\f)$ we have $\int_{gK_{\O^\m}g^{-1}} |\langle \pi(h)\phi, \phi \rangle|^2 dh \gg_\eps C_1^{-1-\eps} \langle \phi, \phi \rangle^2$. This is a mild technical condition that is satisfied by several families of automorphic forms, including newforms, automorphic forms corresponding to minimal vectors, $p$-adic microlocal lifts, and so on. For a more down-to-earth but slightly stronger condition, see Remark \ref{rem:local}.}, we have  the following bound: \begin{equation}\label{e:localgen}\| \phi\|_\infty \ll_{\eps}  C_1^{\frac12 + \eps}\|\phi\|_2\end{equation} where $C_1$ is the smallest integer such that $\mathrm{cond}(\pi)|C_1^2$. We call \eqref{e:localgen} the local bound (in the level aspect). A  more refined local bound is given in Remark \ref{rem:local} of this paper, under slightly more restrictive assumptions on $\phi$.

\begin{remark}The quantity $C_1$ is equal to $\cond(\pi \times \tilde{\pi})^{1/2}$. One reason that $C_1=\cond(\pi \times \tilde{\pi})^{1/2}$ shows up in \eqref{e:localgen} is that (when $\pi$ is discrete series) $C_1$ approximately equals the formal degree of $\pi$;  see the calculations in \cite{marsh15,saha-sup-level-hybrid} or \cite[Appendix A]{HN-minimal}.
\end{remark}

The local bound \eqref{e:localgen} is essentially due to Marshall \cite{marsh15}. It seems reasonable to call \eqref{e:localgen} the local bound because  (to quote Marshall in \cite{marsh15}) it appears to be ``the best bound that may be proved by only considering the behaviour of $\phi$ in one small open set at a time,  without taking the global structure of the space into account".
We note that the bound \eqref{e:localgen} is also true in the non-compact setting of automorphic forms on $\GL_2(\A)$, provided one restricts the domain of $\phi$ to a fixed compact set. It seems worthwhile here to comment on the relationship between the local bound \eqref{e:localgen} and the trivial bound \eqref{e:trivial}. It can be shown easily that the local bound \eqref{e:localgen} is always \emph{at least as strong} as the trivial bound \eqref{e:trivial}. However, these two bounds have somewhat different flavours: the trivial bound applies to \emph{all} forms that transform by unitary characters of compact subgroups of a particular volume (and does not depend on the conductors of the associated representations) while the local bound applies to forms whose associated representations have a particular \emph{conductor} (and does not depend on some choice of  subgroup that transforms the form by a character).

A central problem in this field  (which also generalizes to higher rank automorphic forms) is to improve upon the local bound \eqref{e:localgen} for natural families of automorphic forms $\phi$. An obvious strategy to try to do this would be to first improve upon the trivial bound for some class of subgroups (as we do in Theorem A in wide generality), and then use this result (for some carefully chosen subgroup) to improve upon the local bound. This naive strategy works best when the local component $\phi_p$ for each ramified prime $p$ is an eigenvector of a relatively large neighbourhood of the identity.  A key class of $\phi_p$ for which this is true is the family of \emph{minimal vectors}. Minimal vectors may be viewed as
$p$-adic analogues of holomorphic vectors at infinity and have several remarkable properties, which were proved in our recent work \cite{HNS} (where the analytic properties of the corresponding global automorphic forms of minimal type were studied for the first time in the setting of $\GL_2$).  The main result of \cite{HNS} proved an optimal sup-norm bound for such forms in the setting of $\GL_2$.

However, the techniques used in \cite{HNS} relied on a very special property of the Whittaker/Fourier expansion of $\phi$ which only works in the non-compact setting. Therefore, the proof cannot carry over to the compact case, i.e., to our case of  indefinite quaternion division algebras $D$, as no Whittaker/Fourier expansions exist here. A major impetus behind this paper was to improve upon \eqref{e:localgen} for automorphic forms of minimal type on \emph{compact arithmetic surfaces}. One consequence of Theorem A is that we can now do this.

\medskip

\textbf{Theorem B.} (see Theorem \ref{t:main})\emph{ Let $\pi$ be an irreducible, automorphic
  representation of $D^\times(\A)$ with trivial central character whose local component at each prime dividing $d$ is one-dimensional, and let
  $C$ denote the (arithmetic) conductor of $\pi$. Assume that $C$ is the fourth power of
  an odd integer and suppose, for each prime $p$ dividing
  $C$, that $\pi_p$ is a supercuspidal representation. Let
  $\phi$  in the space of
  $\pi$ correspond to a minimal vector at each prime dividing $C$, a spherical vector at all other primes, and a vector of minimal weight at infinity. Then
\begin{equation}\label{e:localminimal}\|\phi \|_\infty \ll_{D,\pi_\infty, \eps} C^{\frac16+ \eps}\|\phi\|_2.\end{equation}}

We remark that the condition on the conductor being the fourth power of an odd integer  is a  convenient one that was assumed for the definition of minimal vectors in our work \cite{HNS}. However, this restriction  has been partially removed in more recent work of Hu and Nelson \cite{HN-minimal} where they define and study properties of minimal vectors for all supercuspidal representations of $D^\times$ where $D$ is a (split or division) quaternion algebra over a $p$-adic field with $p\ne 2$. Using their work, we prove a more general version of Theorem B (Theorem \ref{t:main}) that applies to \emph{any} odd conductor $C$.

The quantity $C^{1/6}$ on the right side of \eqref{e:localminimal}  represents  one-third of the progress from the local bound $C^{1/4}$ extracted from the right side of \eqref{e:localgen} (we observe that in the setting of Theorem B, $C_1 = C^{1/2}$) to the conjectured\footnote{To be fair, not a lot of evidence exists for this conjecture beyond the fact that it the best possible bound one can hope to prove, and no theoretical obstructions to achieving it have been found.} true bound of $C^\eps$.  Theorem B therefore gives a Weyl-type exponent, which appears to be a  natural limit for the problem of improving upon the local bound with current tools, at least in cases where no Fourier expansions are available.

Theorem A also leads to a sub-local bound in certain other settings. These other settings include the case of ``$p$-adic microlocal lifts" in the sense of \cite{nelson-padic-que}. The $p$-adic microlocal lifts may be naturally viewed as the principal series analogue of minimal vectors. Indeed, for the corresponding global automorphic forms, we are also able to prove a Weyl strength sub-local bound, see \eqref{e:microlocal}. We also obtain a bound for \emph{newforms} that generalizes and strengthens  previously known results; see Theorem \ref{t:newforms}. \footnote{A much stronger bound in the setting of newforms of trivial character in the \emph{depth aspect} will appear in  a sequel to this paper.}

Finally, we remark that the results of this paper appear to be the first time that the local bound in the conductor aspect has been improved upon for squarefull conductors, for any kind of automorphic form on a compact domain. (In the non-compact case, this had been achieved in our previous paper \cite{HNS}.) It seems also worth mentioning here the very recent work of Hu \cite{Hu:minimal} which generalizes \cite{HNS} and proves a sub-local bound in the depth aspect for automorphic forms of minimal type on $\GL_n$ under the assumption that the corresponding local representations have ``generic" induction datum.

\subsection{Key ideas}\label{s:introkeyideas} The heart of this paper is our solution to a counting problem for the number of elements that are ``close" to the identity inside a (varying) quaternion order. This counting problem  arises naturally in the amplification method for the level-aspect sup-norm problem. Roughly speaking, given an order $\O \subseteq \O^\m$ of $D$, we are interested in good upper bounds for the integer $$N_\O(m;z)  =\left|\{\alpha \in S_\O(m): u(z , \alpha z) \ll 1\}\right|$$ where $$S_\O(m) = \{ \gamma \in \iota_\infty(\O), \ \det(\gamma)=m\},$$ and $u(z_1, z_2)$ denotes the hyperbolic distance.

 Above, $z$ is any point on the upper-half plane $\H$. Note, however, that for any $\gamma \in \Gamma_{\O^\m}$, we have $N_\O(m;z)  = N_{\O '}(m;\gamma z)$, for an order $\O'$ that is \emph{conjugate} to $\O$ by an element of $(\O^\m)^\times$. This allows us to move $z$ to a \emph{fixed compact set} $\J$, namely $\J$ equal to some choice of fundamental domain for the action of $\Gamma_{\O^\m}$ on $\H$, at the cost of changing the order $\O$ to a suitable $(\O^\m)^\times$-conjugate of it. Now suppose that for each $m$ and $z \in \J$, we are able to prove a bound for $N_\O(m;z)$  that depends only on $m$, $\J$, and the $(\O^\m)^\times$-conjugacy class of $\O$.  Then we  have actually proved a bound that is valid for all $z\in \H$. This reduction is a key piece in our strategy and can be viewed as a workaround for the situation when $\Gamma_\O$ is not a normal subgroup of $\Gamma_{\O^\m}$ (c.f. the comments in Section 1.3 of \cite{templier-sup}).

In fact we prove two separate bounds for $N_\O(m;z)$ for $z \in \J$. The primary one among them (Proposition \ref{t:counting})  is valid for general \emph{lattices} $\O$ (and does not use the multiplicative structure of the order at all!). The analysis behind the proof of this proposition, carried out in Sections \ref{s:raphael} - \ref{s:proofcount},   may be of independent interest. Roughly speaking, we use a workhorse lemma on small linear combinations of integers to reduce the counting problem to several elementary ones involving simple linear congruences. The reader may wonder at this point why we do not instead use standard lattice counting results such as  Proposition 2.1 of \cite{bhw93}. The reason is that those counting results are typically in terms of the successive minima of the lattice, which is not a preserved quantity for $(\O^\m)^\times$-conjugates of the lattice. In contrast, our method allows us to obtain a strong upper bound (Proposition \ref{t:counting}) in terms of the \emph{invariant factors} of the lattice in $\O^\m$ (the invariant factors are the same for all $(\O^\m)^\times$-conjugate lattices).

However, the bound obtained in Proposition \ref{t:counting} is sufficient for our purposes only when the lattice is \emph{balanced} in the sense that its largest invariant factor is not very large (relative to the level).
 This raises the question: how do we deal with unbalanced lattices? For this, we observe another useful fact: the sup-norm of an automorphic form $\phi$ does not change when the form is replaced by some $D^\times(\A)$-translate of it. Now, given $\phi$ as in Theorem A, a $D^\times(\A)$-translate of $\phi$ transforms by a character of an order $\O'$ that is \emph{locally isomorphic} to (in the same genus as) the order $\O$ that we started off with.  This leads us to investigate which orders have the key property of having a locally isomorphic order whose largest invariant factor is not very large. We solve this problem by a careful case-by-case analysis (see Section \ref{s:orderbalanced}) relying on the explicit classification of local \emph{Gorenstein orders} due to Brzezinski. The answer (essentially) is that any order of level $N$ is locally isomorphic to an order whose largest invariant factor divides $N_1$. This result may be of independent interest.

The upshot of all this is that the \emph{only }orders for which our general lattice counting result (Proposition \ref{t:counting}) does not lead to a non-trivial  sup-norm bound are those whose levels are close to squarefree. To deal with this remaining case, we follow Templier's method \cite{templier-sup, templier-sup-errata}  and prove a second counting result (Proposition \ref{t:countingnew}), which uses the ring structure of the order to prove that
elements that are close to the centralizer of some point must lie in a quadratic subfield. The combination of these two counting results lead directly to Theorem A above, and explain the shape the bound therein takes. Indeed, the term $\max(N^{1/3}, N_1^{1/2})$ in Theorem A comes from our first counting result, while the term $N^{11/24}$ comes from our second counting result.

Once the counting results are in place, we feed it into the amplification machinery to prove our main Theorems, closely following the adelic language employed in our previous paper \cite{saha-sup-level-hybrid}. It might be worth mentioning here that we use the slightly improved amplifier introduced by Blomer--Harcos--Mili\'cevi\'c in \cite{blomer-harcos-milicevic} rather than the amplifier used in works like \cite{harcos-templier-2,saha-sup-level-hybrid}, which saves us some technical difficulties.
\subsection{Additional remarks}\label{s:introadd}
 For simplicity, we have only proved a level aspect bound in Theorem A. It should be possible to extend the argument to prove a non-trivial hybrid bound, however we do not attempt to do so here. It may also be possible to extend some parts of this paper (with additional work) to the case of number fields, and to certain higher rank groups. This is because our counting argument for general lattices is elementary and highly flexible, and should generalise to anisotropic lattices of higher rank.

 We end this introduction with a final remark. As explained in Section \ref{s:localbdintro} of this paper, our main result leads to an improvement of the local bound \eqref{e:localgen} in the level aspect for various families of automorphic forms, particularly the ones of minimal type studied in \cite{HNS, HN-minimal}. This uses crucially the fact that minimal vectors in supercuspidal representations $\pi_p$ are eigenvectors for the action of a relatively large subgroup (having volume around $\cond(\pi_p)^{-1/2}$).   In contrast, newvectors in $\pi$ behave well only  under the action of a much smaller subgroup  (having volume around $\cond(\pi_p)^{-1}$). Therefore, the approach of this paper does not immediately lead to an improvement over the local bound for newforms in the depth aspect where the conductor varies over powers of a fixed prime. However, all hope is not lost -- it turns out that one can augment this approach with suitable results on \emph{decay of matrix coefficients}. This is the topic of ongoing work of the author with Yueke Hu, which will be published in a sequel to this paper. Our method there will allow us to beat the local bound for newforms in the depth aspect for the first time.

\subsection*{Acknowledgements}I would like to thank Yueke Hu and Paul Nelson for useful discussions, and Raphael Steiner for his generous help with the proof of Lemma \ref{l:key}. I would also like to thank the anonymous referee for his suggestion to rephrase the argument of Section \ref{s:proofcount} in terms of matrix manipulations and for other suggestions which have improved the readability of this paper.

\section{A counting problem for lattices} \label{s:counting}
\subsection{A lemma on linear combinations of integers} \label{s:raphael}
The object of this subsection  is to prove an elementary but very useful lemma on the existence of ``small" linear combinations of integers coprime to another integer. It is possible that some version of this lemma has appeared elsewhere, but we were unable to find a suitable reference. The proof below is essentially due to Raphael Steiner (private correspondence, March 2018) and we are grateful to him for his help and for allowing us to include his proof here.

\begin{lemma}\label{l:key}For any $\epsilon > 0$, and any $c\ge 1$, there is a positive constant $C_{c,\epsilon}$
such that for all $(n+2)$-tuples of integers $(a_0, a_1, \ldots, a_n, N)$, with $N>0$,  $\gcd(a_0, a_1, \ldots, a_{n}, N)=1$, there
exists at least $c^n$ distinct $n$-tuples of integers $(p_1, p_2, \ldots, p_n)$ with \begin{enumerate}
\item \label{1} $0 \le p_i \le C_{c,\eps} N^\epsilon$ for $1 \le i \le n$,
\item $\gcd(a_0 + a_1p_1 + a_2p_2 + \ldots +a_np_n, N)=1$.
\end{enumerate} More precisely, for any non-empty subset $I$ of $\{1, \ldots, n\}$, let $P_I$ denote the natural projection map from $\Z^n$ to $\Z^I$ taking any $n$-tuple to its associated sub-tuple corresponding to the indices in $I$ . Then, if $S \subset \Z^n$ consists of all the $n$-tuples $(p_1, p_2, \ldots, p_n)$  satisfying the conditions (1), (2) above, then $|P_I(S)| \ge c^{|I|}$ for each subset $I$ of $\{1, \ldots, n\}$.
\end{lemma}
\begin{remark}With more sophisticated sieving methods a la \cite{iwan78}, one can replace $N^{\epsilon}$ by $\log(N)^2$ in the lemma above.
\end{remark}
\begin{remark}  We encourage the reader to initially focus on the case $n=1$  of the lemma above to get a feel for the statement. In this paper, we will need the lemma only in the case $c=2$, $n=2$.
\end{remark}
\begin{proof} We may assume without loss of generality that $\gcd(a_0, a_1, \ldots, a_{n})=1$ and that $N$ is squarefree. Indeed, if these conditions are not met, we can replace each $a_i$ by $a_i/d$ where $d = \gcd(a_0, a_1, \ldots, a_{n})$ and we can replace $N$ by its largest squarefree divisor, so that this modified setup does satisfy the conditions. Any constant $C_{c,\eps}$ that works for this modified setup will also work for the original setup.

 We will prove the lemma by induction on $n$. Let us prove the base case $n=1$. The starting point for this case is the following  fact:\emph{ For all $\epsilon>0$, there is a constant $D_\eps$ such that for all positive integers} $a_0, a_1, Q$ with $\gcd(a_1, Q)=1$ and $X>0$ we have
\begin{equation}\label{e:case2} \sum_{\substack{1 \le m \le X \\ \gcd(a_0+ma_1,Q)=1}} 1 \ \ge \  \frac{\phi(Q)}{Q}X \  -  \ D_\eps Q^{\eps/2}.\end{equation}

The proof of \eqref{e:case2} follows from the following calculation: \begin{align*}&\sum_{\substack{1 \le m \le X \\ \gcd(a_0+ma_1,Q)=1}} 1 = \sum_{1 \le m \le X} \  \sum_{\substack{d | Q \\ d|a_0+ma_1}} \mu(d) = \sum_{d | Q} \mu(d) \sum_{\substack{1 \le m \le X \\ m \equiv -\overline{a_1} a_0 \mod(d)}} 1 \\ &= \sum_{d | Q} \mu(d)\left(\frac{X}{d} + O(1)\right) =
 X \sum_{d|Q} \frac{\mu(d)}{d} + \sum_{d|Q} O(1) = \frac{\phi(Q)}{Q} X + O_{\epsilon}(Q^{\epsilon}).\end{align*}

 Let us now explain how the case $n=1$ of the lemma follows from \eqref{e:case2}. We may assume that $D_\eps >1$. We can find a constant $E_\eps$ such that $$\frac{Q}{\phi(Q)} \le E_\eps Q^{\eps/2}$$ for all $Q$ and all $\eps>0$. Now put $Q = N/\gcd(a_1, N)$ and choose $C_{c,\eps} > (c+1) D_\eps E_\eps$. Then picking $X = C_{c,\eps} Q^\eps$ in \eqref{e:case2} we have that

\begin{equation}\label{e:case3} \sum_{\substack{1 \le m \le  C_{c,\eps} Q^\eps \\ \gcd(a_0+ma_1,Q)=1}} 1 \ \ge \  \frac{\phi(Q)}{Q} C_{c,\eps} Q^\eps\  -  \ D_\eps Q^{\eps/2} \ge c D_\eps Q^{\eps/2} \ge c.\end{equation}

So there exist at least $c$ distinct integers $m_1, m_2, \ldots, m_c$, such that for $p_1 \in \{m_1, m_2, \ldots, m_c\}$, we have $$p_1 \le C_{c,\eps} N^\eps, \quad \gcd(a_0+p_1a_1,N/\gcd(a_1, N))=1.$$ However, since $\gcd(a_0, a_1)=1$, we have that $$\gcd(a_0+p_1a_1, \ N/\gcd(a_1, N))=1 \Rightarrow \gcd(a_0 +p_1a_1, N)=1.$$ The proof of the case $n=1$ of the Lemma is complete.

We now prove the induction step. Assume that the lemma is proved for $n=k$. We will proe the lemma for $n=k+1$. Suppose we have  integers $a_0, a_1, \ldots, a_{k+1}, N$, with  $\gcd(a_0, a_1, \ldots, a_{k+1})=1$.  By replacing $a_{k+1}$ by its residue modulo $N$ if necessary, we may assume that $0 \le a_{k+1} \le N$.

Since $a_{k+1} \le N$, by the case $n=k$ of the Lemma,  we can find a set $S_k\subset \Z^k$ with the following properties:

 \begin{itemize}
 \item For each $(p_1, \ldots, p_k) \in S_k$, we have $0 \le p_i \le C_{c,\eps} N^\epsilon$ and
  $\gcd(a_0 + a_1p_1 + \ldots +a_kp_k, a_{k+1})=1$.
 \item For each non-empty subset $I$ of $\{1, \ldots, k\}$, $|P_I(S)| \ge c^{|I|}$.
  \end{itemize}

   We now construct a set $S_{k+1} \subset \Z^{k+1}$. Namely, given any $\tilde{s} = (p_1, \ldots, p_k) \in S_k$, we use the case $n=1$ of the lemma to find $c$ distinct  integers $p_{\tilde{s},i}$ for $1\le i \le c$, such that $0 \le p_{\tilde{s},i} \le C_{c,\eps} N^\epsilon$ and such that $$\gcd((a_0 + a_1p_1 + \ldots +a_kp_k) + a_{k+1}p_{\tilde{s},i}, N)=1.$$ Define $$S_{k+1} = \{(\tilde{s},p_{\tilde{s},i}):  \tilde{s}  \in S_k, \ 1\le i \le c\}.$$ It is clear that $S_{k+1}$ satisfies the required conditions and thus the induction step is complete.
\end{proof}

\subsection{Lattices in quaternion orders}\label{s:quatfirst}   Let $D$ be an indefinite quaternion division algebra over $\Q$. We let $d$ denote the reduced discriminant of $D$, i.e., the product of all primes such that $D_p$ is a division algebra.  Fix once and for all a maximal order $\O^\m$ of $D$, and an isomorphism\footnote{Such an isomorphism $\iota_\infty$ is unique up to conjugation by $\GL_2(\R)$.} $\iota_\infty: D_\infty \rightarrow M(2, \R)$.\begin{comment}For each place $v \notin S$, fix an isomorphism $\iota_v: D_v  \xrightarrow {\cong} M(2,\Q_v)$. We assume that these isomorphisms are chosen such that for each finite prime $p \notin S$, we have $\iota_p(\O_p) = M(2,\Z_p)$. It is well known that all such choices are conjugate to each other by some matrix in $\GL_2(\Z_p)$.\end{comment}

\begin{comment}We fix two integers $P$, $Q$ such that $P>0$, $Q<0$, and $D$ is isomorphic to $\left(\frac{P,Q}{\Q}\right)$. So we have $D = \Q + \Q I + \Q J + \Q IJ$ with $I^2=P$, $J^2=Q$, $IJ=-JI$. We can then explicitly fix the isomorphism $\iota_\infty: D_\infty \rightarrow M(2, \R)$ as follows:
$$\iota_\infty(a+bI+cJ+dIJ) = \mat{a+b\sqrt{P} \ }{cQ+dQ\sqrt{P}}{c-d\sqrt{P}}{a-b\sqrt{P}}.$$\end{comment}

For $\alpha \in D$, let $\alpha \mapsto \overline{\alpha}$ be the standard involution of $D$ and let the reduced norm $\nr$ and trace $\tr$ be given by $$\nr(\alpha) = \alpha \overline{\alpha}, \quad \tr(\alpha) = \alpha + \overline{\alpha}.$$ Given a subset $\L$ of $\O^\m$, and an integer $m$, we define $$\L(m) =\{\alpha \in \L: \nr(\alpha) = m\},$$ $$\L_m =\{\alpha \in \L: \tr(\alpha) = m\}.$$ Thus, $\O^\m_0$ denotes the trace 0 elements of $\O^\m$, and $\O^\m(1)$ is the subgroup of $(\O^\m)^\times$ with reduced norm 1. We fix, once and for all, three elements $i_1$, $i_2$, $i_3$ in $\O^\m_0$ such that $$\O^\m_0 = \Z i_1 \oplus \Z i_2 \oplus \Z i_3.$$ So we have $$D = \Q + \Q i_1 + \Q i_2 + \Q i_3.$$

Fix a compact subset $\J$ of $\H$.\footnote{Later in this paper, we will take  $\J$ to be a fundamental domain for the action of $\iota_\infty(\O^\m(1))$ on $\H$.}
Given a subset $\L$ of $D$, and an element $z \in \H$, $\delta>0$, define for each positive integer $m$ the set
$$\L(m;z, \delta)  =\{\alpha \in \L(m): u(z, \iota_\infty(\alpha) z) \le \delta\}.$$

The reader should think of $\delta \asymp 1$ as fixed, since all constants will be allowed to depend on $\delta$ (in fact, for  our eventual applications, we will actually fix $\delta=1$, however it will be useful to keep this variable $\delta$ around for now). Our goal is to bound the cardinality of $\L(m;z, \delta)$ (in terms of some basic invariants of $\L$) whenever $\L \subseteq \O^\m$ is a \emph{lattice} containing 1.

Let $\L\subseteq \O^\m$ be a lattice containing 1. We denote $$N = [\O^\m : \L]$$ and call $N$ the level of $\L$. By the structure theorem for finitely generated abelian groups, the finite group $\O^\m_0/L_0$ is isomorphic to $(\Z/M_1\Z) \times (\Z/M_2\Z) \times (\Z/M_3\Z)$ for some uniquely defined positive integers $M_1|M_2|M_3$, which are sometimes known as invariant factors. We will refer to these integers as the shape of $\L$.

\begin{definition}Given positive integers $M_1, M_2, M_3$ such that $M_1|M_2|M_3$, a lattice $\L$ of $D$ is said to have shape $(M_1,M_2, M_3)$ if $1 \in \L$, $\L \subseteq \O^\m$ and there exist elements $\Delta_1, \Delta_2, \Delta_3 \in \O^\m_0$ such that:

\begin{enumerate}
\item $\O^\m_0 =  \Z\Delta_1 \oplus \Z\Delta_2 \oplus\Z\Delta_3$,
\item $\L_0 =  M_1\Z\Delta_1 \oplus M_2\Z\Delta_2\oplus M_3\Z\Delta_3$.
\end{enumerate}

Furthermore, we denote $M=M_1M_2M_3$ and call it the level of $\L_0$.
\end{definition}

\begin{remark}Let  $\L \subseteq \O^\m$ be a lattice of shape $(M_1,M_2, M_3)$ and level $N$. If $x\in D$ satisfies $x \L x^{-1} \subseteq \O^\m$, then one may ask about the shape and level of $\L':= x \L x^{-1}$.

It is easy to see that $\L'$ always has level $N$ but its shape might be different in general. However, if $x \in (\O^\m)^\times$, then $\L'$ also has shape $(M_1,M_2, M_3)$.
\end{remark}

It turns out to be more convenient for us to descend to the sublattice $\Z \oplus \L_0$, for which the next lemma is essential.

\begin{lemma}\label{l:asymp}Let $\L$ be a lattice in $D$ such that $\L \subseteq \O^\m$ and $1 \in \L$. Then $$[\L : \Z \oplus \L_0] \in \{1,2\}.$$
\end{lemma}
\begin{proof}
Given an element $\ell \in \L$, we have $\tr(\ell) \in \Z$ and furthermore, $\ell$ belongs to $\Z \oplus \L_0$ if and only if $\tr(\ell) \in 2\Z$. So if $\ell_1$ and $\ell_2$ are two elements in $\L$, neither of which belong to $\Z \oplus \L_0$, then $\ell_1 + \ell_2 \in \Z \oplus \L_0$. The statement follows.
\end{proof}

Let $\L$ be a lattice as in Lemma \ref{l:asymp}, and let $N$ be the level of $\L$, and $M$ the level of $\L_0$. Using Lemma \ref{l:asymp} and the fact that $\O^\m_0 + \Z$ has index two in $\O^\m$, we observe that \begin{equation}\label{e:card}M = [\O^\m_0: \L_0] =  \frac{[\O^\m: \Z \oplus \L_0]}{2} = \frac{N}e, \quad e \in \{1,2\} \end{equation} where $e$ equals 2 if and only if $\L = \Z \oplus \L_0$. So $N \asymp M=M_1M_2M_3$.

\begin{remark}\label{rem:shapeinv} Consider the lattice $\L$ in $O^\m$. The invariant factors of $\L$ with respect to $\O^\m$ are $(1, M_1', M_2', M_3')$, for some integers $M_1'|M_2'|M_3'$.
Now, using Lemma \ref{l:asymp}, we obtain for $i=1,2,3$, \begin{equation}\label{e:card2}M'_i  = e_i M_i, \quad e_i\in \{1, 2\},  \quad e_1e_2e_3 = e \in \{1,2\}. \end{equation}\end{remark}

We now state our main counting result.

\begin{proposition}\label{t:counting} Let $\L \subseteq \O^\m$ be a \emph{lattice} containing 1. Suppose that $\L$ has shape $(M_1, M_2, M_3)$ and level $N$. Let $z \in \J$ and $1 \le L \le N^{O(1)}$. Then the following hold.

\begin{equation}\label{eq1}\sum_{1\le m \le L } |\L(m;z, \delta)| \ll_{\epsilon, \delta} N^{\epsilon} \left(  L^{1/2} + \frac{L}{M_1} +\frac{L^{3/2}}{M_1M_2} +\frac{L^2}{N} \right).\end{equation}
 \begin{equation}\label{eq2}\sum_{1\le m \le L } |\L(m^2;z, \delta)| \ll_{\epsilon, \delta} N^{\epsilon} \left(L +\frac{L^{2}}{M_1M_2} +\frac{L^3}{N}   \right).\end{equation}

\end{proposition}
\begin{remark}The constants implicit in the bounds above depend only  on $\eps, \delta$ and the \emph{fixed} objects $D, \O^\m, \iota_\infty, \J$. \end{remark}

\begin{remark}Note that the bounds do \emph{not} depend on the elements $\Delta_1$, $\Delta_2$, $\Delta_3$. Hence the bounds obtained are \emph{uniform} over all $(\O^\m)^\times$-conjugates of $\L$. This will be key for us later on.
\end{remark}

\begin{remark}Because of \eqref{e:card} one can replace $N$ by $M$ in the theorem above, if one wishes. Furthermore, because of \eqref{e:card2}, one can replace $M_1$, $M_2$ in the theorem above by $M_1'$, $M_2'$ respectively, if one wishes to.
\end{remark}

\begin{remark}The bound obtained in Proposition \ref{t:counting} is not sufficient for our purposes when $M_1M_2$ is small in relation to $N$. So in Section \ref{s:countingorder}, we will prove another counting result under the additional assumption that $\L$ is an \emph{order}.
\end{remark}

 \subsection{Proof of Proposition \ref{t:counting}}\label{s:proofcount}
\begin{lemma}\label{l:compactdisc}For any $\delta>0$, there exists a constant $T$ (depending on $\delta$, $\J$ and $\iota_\infty$) with the following property: For $m>0$, $z \in \J$, and $\alpha = a_0 + a_1i_1 + a_2ei_2+ a_3i_3\in D(m)$ satisfying $u(z, \iota_\infty(\alpha)z) \le \delta$,  we have $$|a_i| \le T m^{1/2} \text{ for } i=0,1,2,3.$$
\end{lemma}
\begin{proof}It suffices to show that the set of all $\frac{\alpha}{\sqrt{m}}$ as above lies in a compact set depending only on $\delta$, $\J$ and $\iota_\infty$. The subset $\Omega_{\J, \delta}$ of $\SL_2(\R)$ given by $$\Omega_{\J, \delta}=\{\gamma \in \SL_2(\R): u(z, \gamma z) \le \delta \text{ for all }z \in \J\}$$ is compact since the stabilizer of each point is compact, and $\J$ is compact. Therefore the subset $ \iota_\infty^{-1}(\Omega_{\J, \delta})$ is a compact subset of $D(1)$ that contains all the elements $\frac{\alpha}{\sqrt{m}} = \frac{a_0 + a_1i_1 + a_2i_2+ a_3i_3}{\sqrt{m}}$ as in the Lemma. The result follows.
\end{proof}

We now prove Proposition \ref{t:counting}. Using Lemma \ref{l:asymp}, we may assume that $\L=\Z \oplus \L_0$. Indeed, putting $\L' = \Z \oplus \L_0$, we see that $$|\L'(m;z,\delta)| \le |\L(m;z,\delta)| \le |\L'(4m;z,\delta)|.$$ So by shrinking $\L$ if necessary, we will assume throughout the rest of this subsection that \begin{equation}\label{e:assumplattice}\L = \Z \oplus \L_0.\end{equation} Now, using Lemma \ref{l:compactdisc}, we see that Proposition \ref{t:counting} would follow from the following statement:

\medskip

\emph{Let $\L = \Z \oplus M_1\Z\Delta_1 \oplus M_2\Z\Delta_2\oplus M_3\Z\Delta_3$ be a lattice of shape $(M_1, M_2, M_3)$ and level $N$, where  $\Delta_1, \Delta_2, \Delta_3 \in \O^\m_0$ and $M_1|M_2|M_3$.  For $T>0$,  define $$\L(m, T) = \{ a_0 + a_1i_1 + a_2i_2+ a_3i_3\in \L(m): a_i \in \Z, \ |a_i| \le T m^{1/2} \text{ for } 0\le i \le 3\}.$$ Then for $1 \le L \le N^{O(1)}$ we have:}
\begin{equation}\label{eq1new} \#\bigg(\bigcup_{1 \le m \le L}\L(m, T) \bigg) \ll_{\epsilon, T} N^{\epsilon} \left(  L^{1/2} + \frac{L}{M_1} +\frac{L^{3/2}}{M_1M_2} +\frac{L^2}{N} \right).\end{equation}
 \begin{equation}\label{eq2new}\#\bigg(\bigcup_{\substack{1 \le m \le L \\ \sqrt{m} \in \Z}}\L(m, T) \bigg) \ll_{\epsilon, T} N^{\epsilon} \left(L^{1/2} +\frac{L}{M_1M_2} +\frac{L^{3/2}}{N}   \right).\end{equation}

  We now begin the proof of the bounds \eqref{eq1new}, \eqref{eq2new}. This will complete the proof of Proposition \ref{t:counting}. For brevity, we drop $T$ from the $\ll$ symbol in the rest of this subsection (so all constants henceforth are allowed to depend on $T$).

 Since  $(\Delta_1, \Delta_2, \Delta_3)$ and $(i_1, i_2, i_3)$ are integral bases  for $\O_0$, it follows that  there exists a matrix $\delta \in \GL_3(\Z)$ such that \begin{equation}\label{e:deltaji} [\Delta_1, \Delta_2, \Delta_3]^t = \delta \ [i_1, i_2, i_3]^t. \end{equation} We  note that \begin{equation}\label{e:det}\det(\delta) = \pm 1.\end{equation} In paticular,
$\gcd(\delta_{1,1}, \delta_{1,2}, \delta_{1,3})=1$. Therefore, using Lemma \ref{l:key}, we fix integers $r_2$, $r_3$, both $\ll_\eps N^\eps$, such that \begin{equation}\label{e:R}R = \delta_{1,1} + r_2 \delta_{1,2} + r_3 \delta_{1,3} \text{ satisfies }  \gcd(R, N)=1.\end{equation}

 Define $R_2$, $R_3$ via the equation  \begin{equation}\label{e:R2R3}\delta \ [1, r_2, r_3]^t = [R,R_2, R_3]^t.\end{equation} Fix an integer $\bar{R}$ such that $R\bar{R} \equiv 1 \pmod{N}.$ Now define
 $$h = \begin{bmatrix}1&0&0\\ - R_2 \bar{R}&1&0 \\ 0&0&1 \end{bmatrix}.$$ Let the integers $S_1$, $S_2$, $S_3$ be defined via the equation $$[0,1,0]h\delta = [S_1, S_2, S_3].$$

 We claim that $\gcd(S_1, S_2, S_3, N)=1$.
Suppose this were not true. Then we would be able to find a prime $p|N$ such that $$[0,1,0] h \delta  \equiv [0,0,0] \pmod{p}.$$  Since $\delta \in \GL_3(\Z)$, this implies that  \begin{equation}\label{e:contra}[0,1,0] h \equiv [0,0,0] \pmod{p}.\end{equation}
However, the left-hand-side of \eqref{e:contra} is a vector of the form $[*, 1 ,0]$, leading to a contradiction. This contradiction proves that $\gcd(S_1, S_2, S_3, N)=1$.

Using Lemma \ref{l:key}, we now fix integers $s_2$, $s_3$, both $\ll_\eps N^\eps$ and such that  $$s_2 \ne r_2, \quad S = S_1 + s_2 S_2 + s_3 S_3 \text{ satisfies } \quad \gcd(S, N)=1.$$ Define the element $S'$ via $$S' = [1,0,0]\delta [1, s_2, s_3]^t.$$ Put $$g = \begin{bmatrix}1&1&0\\r_2&s_2&0\\r_3&s_3&1\end{bmatrix}.$$ By our assumption that $s_2 \neq r_2$, we have $\det(g) \neq 0$.  Moreover, from \eqref{e:R2R3} and the definitions of $h$, $S$, $S'$, we see that \begin{equation}\label{e:needed}h \delta g \equiv \begin{bmatrix}R&S'&\delta_{1,3}\\0&S&-R_2 \bar{R} \delta_{1,3} + \delta_{2,3}\\ * &*&* \end{bmatrix} \pmod{N}.\end{equation}

We now begin the proof of \eqref{eq1new}. Let $$\alpha \in \bigcup_{1 \le m \le L}\L(m, T).$$ Our strategy will be to associate to each such $\alpha$ a quadruple $(a_0, A_\alpha, B_\alpha, a_3)$ such that the function $\alpha \mapsto  (a_0, A_\alpha, B_\alpha, a_3)$ is injective. A bound for the cardinality of $\bigcup_{1 \le m \le L}\L(m, T)$  will then follow by bounding the number of distinct tuples $(a_0, A_\alpha, B_\alpha, a_3)$ that are possible.

Write $$\alpha = a_0 + a_1i_1 + a_2i_2+ a_3i_3.$$ We have that $|a_i| \ll L^{1/2}$. Furthermore, using the definition of $\L$ together with \eqref{e:deltaji}, it follows that there exist integers $u_1, u_2, u_3$ such that
\begin{equation}\label{e:ai} [a_1, a_2, a_3] = [u_1M_1, u_2M_2, u_3 M_3] \delta.\end{equation}
Define $A_\alpha$, $B_\alpha$ via the equation \begin{equation}\label{e:Aalpa}[A_\alpha, B_\alpha, a_3] = [a_1, a_2, a_3] g = [u_1M_1, u_2M_2, u_3 M_3] \delta g.\end{equation}
 We saw already that $g$ is an invertible matrix and therefore the mapping $\alpha \mapsto  (a_0, A_\alpha, B_\alpha, a_3)$ is injective, as required.

It remains to  bound the number of distinct tuples of integers $(a_0, A_\alpha, B_\alpha, a_3)$ that are possible. To achieve that, we will prove certain bounds and congruences satisfied by such tuples. First of all, using \eqref{e:Aalpa} and the fact that $r_2$, $r_3$ are both $\ll_\eps N^\eps$, it follows that
$|A_\alpha| \ll_\eps L^{1/2}N^\eps$ and $A_\alpha \equiv 0 \pmod{M_1}$. \emph{So there are $\ll_\eps N^\eps\left(1 + \frac{L^{1/2}}{M_1}\right)$ choices for $A_\alpha$. Henceforth, assume such a choice has been made.}

Next, using \eqref{e:needed} and \eqref{e:Aalpa} we see that \begin{align}[A_\alpha, B_\alpha, a_3] &= [u_1M_1, u_2M_2, u_3 M_3] h^{-1} h \delta g \nonumber  \\ \label{penul} &\equiv [u_1M_1 + R_2\bar{R}u_2M_2, u_2M_2, 0] \begin{bmatrix}R&S'&\delta_{1,3}\\0&S&-R_2 \bar{R} \delta_{1,3} + \delta_{2,3}\\ * &*&* \end{bmatrix}  \pmod{M_3} \\  \label{ult} &\equiv [u_1M_1, 0, 0] \begin{bmatrix}R&S'&\delta_{1,3}\\ *&*&* \\ *&*&* \end{bmatrix}  \pmod{M_2}.\end{align}

The last congruence \eqref{ult} above gives us $B_\alpha \equiv  A_\alpha \bar{R} S' \pmod{M_2}$. \emph{Since $|B_\alpha| \ll_\eps L^{1/2}N^\eps$, there are $\ll_\eps N^\eps\left(1 + \frac{L^{1/2}}{M_2}\right)$ choices for $B_\alpha$. Henceforth, assume such a choice has been made.}

Furthermore, \eqref{penul} directly leads to an expression of the form $$[u_1M_1 + R_2\bar{R}u_2M_2, u_2M_2] \equiv [A_\alpha \bar{R}, \, \bar{S}\left(B_\alpha -  A_\alpha \bar{R} S' \right)] \pmod{M_3}.$$  Therefore, another application of \eqref{penul} tells us that $a_3$ is known modulo $M_3$ in terms of choices that have already been made. Indeed, if one were to actually perform the above steps, one would arrive at the explicit expression \begin{equation}\label{e:a3} a_3 \equiv  A_\alpha \bar{R} \delta_{1,3} + \bar{S}\left( B_\alpha -  A_\alpha \bar{R} S' \right)(\delta_{2,3} - R_2 \bar{R} \delta_{1,3}) \pmod{M_3}.\end{equation}

\emph{Since $|a_3| \ll_\eps L^{1/2} $ there are $\ll_\eps N^\eps\left(1 + \frac{L^{1/2}}{M_3}\right)$ choices for $a_3$.}

\emph{Finally, as $|a_0| \ll L^{1/2}$, there are $\ll L^{1/2}$ possible choices for $a_0$.}

Combining all the above italicized statements, we see that there are \begin{align*}&\ll_\eps N^\eps L^{1/2} \left(1 + \frac{L^{1/2}}{M_1}\right)\left(1 + \frac{L^{1/2}}{M_2}\right)\left(1 + \frac{L^{1/2}}{M_3}\right)\\& \ll_\eps  N^{\epsilon} \left(  L^{1/2} + \frac{L}{M_1} +\frac{L^{3/2}}{M_1M_2} +\frac{L^2}{N} \right). \end{align*} possible different choices for triples $(a_0, A_\alpha, B_\alpha, a_3)$ as $\alpha$ varies in $\bigcup_{1 \le m \le L}\L(m, T)$. By the injectivity noted earlier, this completes the proof of \eqref{eq1new}.

To prove \eqref{eq2new}, define $\alpha_0 := (\alpha - \overline{\alpha})/2 = a_1i_1 + a_2i_2+ a_3i_3$ . Now, only consider those $\alpha$ such that $m=\nr(\alpha) = \ell^2$ for some $\ell \in \Z$, $0 \le \ell \le L^{1/2}$.  Then we get \begin{equation}\label{e:factor}(a_0 - \ell)(a_0 + \ell) = \alpha_0^2 = -\nr(\alpha_0).\end{equation} Now, $\nr(\alpha_0)$ is an integer and  $|\nr(\alpha_0)| \ll N^{O(1)}$. So, if $\nr(\alpha_0) \neq 0$ then \eqref{e:factor} tells us that there are $\ll_\eps N^\eps$ possibilities for $a_0$. If $\nr(\alpha_0) = 0$ then we must have $\alpha_0 =0$ (since $D$ is a\emph{ division }algebra) and so there are $\ll L^{1/2}$ possibilities for $a_0 = \alpha$. Putting it together, we see that the number of elements in $\bigcup_{\substack{1 \le m \le L \\ \sqrt{m} \in \Z}}\L(m, T)$ is  \begin{align*}& \ll_\eps L^{1/2} + N^\eps \left(  1 + \frac{L^{1/2}}{M_1} +\frac{L}{M_1M_2} +\frac{L^{3/2}}{N} \right). \end{align*} This completes the proof of \eqref{eq2new}.
\subsection{A supplementary counting result for orders}\label{s:countingorder} In this subsection, we give another counting result to supplement Proposition \ref{t:counting}, but one that is applicable only if $\L=\O$ is an\emph{ order}.  \begin{proposition}\label{t:countingnew} Let $\O \subseteq \O^\m$ be an order of level $N$. There is a constant $C$ (depending on $\delta$, $\J$ and $\iota_\infty$) such that for $z \in \J$ and $1 \le m  < CN^{\frac13}$, we have
\begin{equation}\label{eq3} |\O(m;z, \delta)| \ll_{\epsilon, \delta} m^{\eps}.\end{equation}
\end{proposition}

Our proof of Proposition \ref{t:countingnew} is broadly similar to that of Proposition 6.5 of \cite{templier-sup} (see also \cite{templier-sup-errata}).  The proof will follow from the following sequence of lemmas. Throughout the proof, we will use the notations introduced in Section \ref{s:quatfirst} and we will assume (without loss of generality) that \eqref{e:assumplattice} holds.

\begin{lemma}Let $\L$ be a subset of $D$ that is closed under multiplication and contains 1. Let $z\in \H$, $L$ a positive integer, and $\delta >0$. Then the $\Q$-algebra generated by all elements in $\bigcup_{1 \le m \le L}\L(m;z, \delta)$ is contained in the $\Q$-vector-space spanned by $\bigcup_{1 \le m \le L^2}\L(m;z, 2\delta).$
\end{lemma}
\begin{proof}By basic properties of a quaternion algebra, any element of the $\Q$-algebra generated by $\bigcup_{1 \le m \le L}\L(m;z, \delta)$ is a $\Q$-linear combination of elements of the form $\beta=\beta_1\beta_2$ with $\beta_1, \beta_2 \in \bigcup_{1 \le m \le L}\L(m;z, \delta)$. So it suffices to show that any such $\beta$ belongs to $\bigcup_{1 \le m \le L^2}\L(m;z, 2\delta).$ This is clear as $\nr(\beta) = \nr(\beta_1)\nr(\beta_2)  \le L^2$ and $$u(z , \iota_\infty(\beta)z) = u(\iota_\infty(\beta_1^{-1})z,  \iota_\infty(\beta_2)z) \le u(z , \iota_\infty(\beta_1)z) + u(z , \iota_\infty(\beta_2)z) \le 2 \delta.$$
\end{proof}

\begin{lemma}Let $\L$ be a lattice in $D$ of level $N$, $z\in \J$, $L$ a positive integer, and $\delta >0$. Then there is a constant $C$ (depending on $\delta$, $\J$ and $\iota_\infty$) such that the  $\Q$-vector-space spanned by $\bigcup_{1 \le m \le L^2}\L(m;z, 2\delta)$ is proper whenever $L < C N^{1/3}$.
\end{lemma}
\begin{proof}Let $\alpha^{(1)}$, $\alpha^{(2)}$, $\alpha^{(3)}$ be three arbitrary elements of  $\bigcup_{1 \le m \le L^2}\L(m;z, 2\delta)$. It suffices to show that $1$,  $\alpha^{(1)}$, $\alpha^{(2)}$, $\alpha^{(3)}$ are linearly dependant. For $i=1,2,3$, write $$\alpha^{(i)} = a_0^{(i)} + a^{(i)}_1i_1 + a_2^{(i)}i_2+ a_3^{(i)}i_3.$$ Let $A$ be the 3 by 3 matrix whose $(i,j)$th entry is $a_j^{(i)}$ for $1 \le i,j \le 3$. We need to show that $\det(A)=0$.  Let $\L$ have shape $(M_1, M_2, M_3)$ with $M_1M_2M_3 \asymp N$, and let the integers $\delta_{i,j}$ be as in \eqref{e:deltaji}. Therefore we have integers $u_j^{(i)}$ such that for $1 \le i,j \le 3$, we have
\begin{equation} a_j^{(i)} = u_1^{(i)}M_1 \delta_{1,j} + u_2^{(i)}M_2 \delta_{2,j} + u_3^{(i)} M_3 \delta_{3,j}.\end{equation}
Writing the above system of equations in matrix form, we see immediately that $M_1M_2M_3$ divides $\det(A)$. On the other hand, by Lemma \ref{l:compactdisc}, we have that $\det(A) \ll L^{3}$ where the implied constant depends on $\delta$, $\J$ and $\iota_\infty$. So if $L^3 \ll N$, we must have $\det(A) = 0$, as desired.
\end{proof}

Now let $\O \subseteq \O^\m$ be an order of level $N$. The above two lemmas imply that if $1\le m <CN^{1/3}$, then all elements of $\O(m;z, \delta)$ lie in a \emph{quadratic field} (since the only proper $Q$-algebras in a given quaternion algebra are $\Q$ and various embedded quadratic fields). Now the proof of Proposition \ref{t:countingnew} follows from the following lemma and the fact that $d(m) \ll_\eps m^\eps$.

\begin{lemma}Let $F \subset D$ be a quadratic field. Then for any $z \in \H$, and any positive integer $m$, we have \begin{equation}\label{eq3new} |F(m;z, \delta)| \ll_{ \delta} d(m),\end{equation} where $d(m)$ denotes the divisor function and the implied constant is independent of $F$.
\end{lemma}
\begin{proof} Any element of $F(m;z, \delta)$ is a product of a unit in $\O_F^\times$ of norm 1, and an element of $\O_F$ of norm $m$, with the latter taken from a fixed set of cardinality $\ll d(m)$. So we only need to consider the action of units. By the proof of Lemma 6.4 of \cite{templier-sup}, the number of norm 1 units $\kappa \in \O_F^\times$ satisfying $u(z, \iota_\infty(\kappa)z) \le \delta$ is $\ll_\delta 1$. This completes the proof. \end{proof}

\section{The main result: Statement and key applications}\label{s:global}
\subsection{Basic notations}\label{s:globalstatement}  We continue to use the notations established in Section \ref{s:quatfirst}, and introduce some new ones below. Let $\f$ denote the finite places of $\Q$ (which we identify with the set of primes) and $\infty$ the archimedean place.  We let $\A$ denote the ring of adeles over $\Q$, and $\A_\f$ the ring of finite adeles.   Given an algebraic group $H$ defined over $\Q$, a place $v$ of $\Q$, a subset of places $U$ of $\Q$, and a positive integer $M$, we denote $H_v:= H(\Q_v)$, $H_U:=\prod_{v \in U} H_v$, $H_M :=\prod_{p|M} H_p$. Given an element $g$ in $H(\Q)$ (resp., in $H(\A)$), we will use $g_p$ to denote the image of $g$ in $H_p$ (resp., the $p$-component of $g$); more generally for any set of places $U$, we let $g_U$ the image of $g$ in $H_U$.

Recall that $D$ is an indefinite quaternion division algebra over $\Q$ with reduced discriminant $d$ and that we have fixed a maximal order $O^\m$. We denote $G=D^\times$ and $G' =PD^\times = D^\times/Z$ where $Z$ denotes the center of $D^\times$.
   For each prime $p$, let $K_p=(\O^\m_p)^\times$ and let $K_p'$ denote the image of $K_p$ in $G'_p$.  Thus, for $p |d$, $K'_p$ has index 2 in the compact group $G'_p$.

 For each place $v$ that is not among the primes dividing $d$, fix an isomorphism $\iota_v: D_v  \xrightarrow {\cong} M(2,\Q_v)$. We assume that these isomorphisms are chosen such that for each finite prime $p\nmid d$, we have $\iota_p(\O_p) = M(2,\Z_p)$. It is well known that all such choices are conjugate to each other by some matrix in $\GL_2(\Z_p)$. By abuse of notation, we also use $\iota_v$ to denote the composition map $D(\Q) \rightarrow D_v \rightarrow M(2,\Q_v)$.

We fix the Haar measure on each group $G_p$ such that $\vol(K_p)=1$. We fix Haar measures on $\Q_p^\times$ such that $\vol(\Z_p^\times)=1$. This gives us resulting Haar measures on each group $G'_p$ such that $\vol(K'_p)=1$.  Fix any Haar measure on $G_\infty$, and take the Haar measure on $\R^\times$ to be equal to $\frac{dx}{|x|}$ where $dx$ is the Lebesgue measure. This gives us a Haar measure on $G'_\infty$. Take the measures on $G(\A)$ and $G'(\A)$ to be given by the product measure.  \begin{comment} which transform by (the same) unitary character of the centre and such that $|\kappa|$, $|\kappa\|$ are compactly supported on $G'_U := \prod_{p \in U} G_p$, we define $$(\kappa \ast \kappa')(g) = \int_{G'_U} \kappa(g^{-1})\kappa'(gh) dg, $$ and for any representation $\sigma$ of $G_U$ with central character equal to that of $\kappa$, we define for any vector $v \in \sigma$, \begin{equation}R(\kappa)v = \int_{G'_U} \kappa(g)(\sigma(g)v)  \ dg. \end{equation}\end{comment}

 For each continuous function  $\phi$ on the space $G(\A)$, we let $R(g)$ denote the right regular action, given by $(R(g)\phi)(h) = \phi(hg)$. If a  continuous function  $\phi$ on $G(\A)$ satisfies that $|\phi|$ is left $Z(\A)G(\Q)$ invariant, define
$$\|\phi\|_2 = \int_{G'(\Q)\bs G'(\A)} |\phi(g)|^2 dg.$$

Note above that $G'(\Q)\bs G'(\A)$ is compact, so convergence of the integral is not an issue.

\subsection{Some facts on orders and their localizations}\label{s:orderbasic}We recall some facts we will need. Proofs of these standard facts can be found, e.g., in \cite{Voight}.

For any lattice $\L\subseteq \O^\m$ of $D$, we get a local lattice $\L_p$ of $D_p$ by localizing at each prime $p$. These collection of lattices satisfy \begin{equation}\label{e:localization} \L = \{g \in D: g_p \in \L_p \text{ for all primes }p\}.\end{equation}  Conversely, if we are given a collection of local lattices $\{\L_p\}_{p \in \f}$, such that $\L_p \subseteq \O^\m_p$ for all $p$ and $\L_p = \O^\m_p$ for all but finitely many $p$, then there exists a unique lattice $\L\subseteq \O^\m$ of $D$ defined via \eqref{e:localization} and whose localizations at primes $p$ are precisely the $\L_p$. We will refer to $\L$ as the global lattice corresponding to the collection of local lattices $\{\L_p\}_{p \in \f}$.

Two orders of $D$ (or of $D_p$) are isomorphic as algebras if and only if they are conjugate by an element of $D^\times$ (respectively, by an element of $D_p^\times$). Given an order $\O$ of $D$, we define a compact subgroup of $G(\A_\f)$ by $$K_\O =\prod_{p \in \f} \O_p^\times.$$

We define the shape and level of an order as in Section \ref{s:quatfirst}. These quantities have obvious local analogues, and so for each order $\O_p \subseteq \O^\m_p$ of $D_p$, we can define its shape $(p^{m_{1,p}}, p^{m_{2,p}}, p^{m_{3,p}})$ and level $p^{n_p}$. It is easy to see  that $\vol(\O_p^\times) \asymp p^{-n_p}$. If $\O$ is the global order of shape $(M_1, M_2, M_3)$ and level $N$ corresponding to the collection of local orders $\{\O_p\}_{p \in \f}$ with shape and level as above, then for $i=1,2,3$ we have: $M_i = \prod_{p \in \f}p^{m_{i,p}}$, $N= \prod_{p \in \f}p^{n_{p}}$. From this it follows that \begin{equation}\label{e:volumeorder}N^{-1-\eps} \ll_\eps \vol(K_\O) \ll_\eps N^{-1+\eps}.\end{equation}

For each $g \in G(\A_\f)$, and an order $\O$ of $D$, we let $g\cdot\O$ denote the order whose localization at each prime $p$ equals $g_p \O_p g_p^{-1}$. An order is said to be locally isomorphic to (in the same genus as) $\O$ if and only if it is equal to $g \cdot \O$  for some $g \in G(\A_\f)$. Note that \begin{equation}\label{e:gcond}g \cdot\O \subseteq \O^\m \quad \large{\Leftrightarrow} \quad g_p \O_p g_p^{-1} \in \O^\m_p \ \forall \ p.\end{equation} Note also that \begin{equation}\label{e:conjugation}g K_{\O}g^{-1} = K_{g\cdot\O}.\end{equation} Given $g$ satisfying \eqref{e:gcond}, the orders $\O$ and $g \cdot \O$ need not be isomorphic or have the same shape, however they always have the same level. However, note that if $k \in K_{\O^\m}$, then  $k\cdot \O$ has exactly the same shape as $\O$.

\subsection{Statement of main result}\label{s:mainresultsubsec}
Let $\pi = \otimes_v \pi_v$ be an irreducible, unitary, cuspidal automorphic
representation of $G(\A)$ where we identify $V_\pi$ with a (unique) subspace of functions on $G(\A)$, so that $\pi(g)$ coincides with $R(g)$ on that subspace. \begin{comment}Assume that if $p \in S$, then $\pi_p$ has a (non-zero) vector fixed by $K_p$. (This implies that $\pi_p$ is one-dimensional for each $p \in S$.)  We let $\c \subset \f$ denote the set of primes for which $\pi_p$ is ramified, i.e., $\pi_p$ does not have a (non-zero) $K_p$-fixed vector. There are two possibilities for $\pi_\infty$. It is either a principal series representation (in which case we put $k=0$), or it is a holomorphic discrete series representation of lowest weight $k$. We let $T$ denote the archimedean parameter of $\pi_\infty$; it can be taken equal to $k$ if $\pi_\infty$ is a holomorphic discrete series, and equal to $1 + |t|$ if $\pi_\infty$ is a principal series representation of the form $|y|^{it} \sgn(y)^{m} \boxplus |y|^{-it} \sgn(y)^{m},$ for $m \in \{0,1\}$.\end{comment}
Given a compact open subgroup $K' = \prod_{p\in \f} K'_p$ of $K_{\O^\m}$ (where each $K'_p$ is a subgroup of $K_p$, with $K'_p = K_p$ for almost all $p$) and a finite dimensional  representation $\rho$ of $K'$, we say that an automorphic form $\phi \in V_\pi$ is of $(K', \rho)$-type if  the right-regular action of $K'$ on $\phi$ generates a representation isomorphic to $\rho$.
Observe that the existence of a form of $(K', \rho)$-type implies that the restrictions of $\rho$ and $\omega_\pi$ to the centre of $K'$ must coincide.

We can now state our main theorem.

\begin{theorem}\label{t:maingen} Let $\O \subseteq \O^\m$ be an order of $D$ of level $N$ and let $\rho$ be a finite dimensional representation of $K_\O$. Let $N_1$ be the smallest positive integer such that $N$ divides $N_1^2$. Let $\pi = \otimes_v \pi_v$ be an irreducible, unitary, cuspidal automorphic
representation of $G(\A)$. Let $\phi\in V_\pi$ be of $(K_\O, \rho)$-type and assume that $\phi$ is of minimal weight at the archimedean place, i.e., \begin{equation}\label{e:minweight}\phi\left( g\left( \iota_\infty^{-1}\mat{\cos(\theta)}{\sin(\theta)}{-\sin(\theta)}{\cos(\theta)} \right) \right)  = e^{ik \theta}\phi(g)\end{equation} where $k = 0$ if $\pi_\infty$ is principal series, and $k$ is the lowest weight if $\pi_\infty$ is discrete series. Then \begin{equation} \label{e:toprove}
\sup_{g \in G(\A)} |\phi(g) | \ll_{D, \pi_\infty, \eps} N^{\eps} \min(\max(N^{1/3}, N_1^{1/2}), N^{11/24}) \dim(\rho)^{1/2}\|\phi\|_2. \end{equation}
\end{theorem}

Theorem A of the introduction is a special case of Theorem \ref{t:maingen}, where we take $\rho$ to be a character. A key flexibility of Theorem \ref{t:maingen} comes from the fact that given $\phi$, one can optimise which order $\O$ to use depending on how much information one has about the dimensions of the representations $\rho$ generated under the action of various $K_\O$.

In certain cases, however, one may only know the dimension under the action of some $K'$ that is \emph{not} of the form $K_\O$. In such cases one can still get a bound by working with any order $\O$ containing $K'$. The following corollary makes this precise.

\begin{corollary}\label{t:maincor} Let $\phi$ be an automorphic form in the space of $\pi$ such that $\|\phi\|_2 =1$ and \eqref{e:minweight} holds. Suppose that $\phi$ is of $(K', \rho)$-type for some compact open subgroup $K'$ of $K_{\O^\m}$ and some  finite dimensional representation $\rho$ of $K'$. Let $\O \subseteq \O^\m$ be any order of $D$ of level $N$ such that $K' \subseteq K_\O$. Let $N_1$ be the  smallest positive integer such that $N$ divides $N_1^2$. Then \begin{equation}
\sup_{g \in G(\A)} |\phi(g) | \ll_{D, \pi_\infty, \eps} N^{\eps} \min(\max(N^{1/3}, N_1^{1/2}), N^{11/24}) \dim(\rho)^{1/2} [K_\O:K']^{1/2}. \end{equation}
\end{corollary}
\begin{proof}Consider the representation $\rho'$ generated by $\phi$ under the action of $K_\O$. Then from elementary considerations, $$\dim(\rho') \le \dim(\rho) [K_\O:K'].$$ Now apply Theorem \ref{t:maingen} using the fact that $\phi$ is of $(K_\O, \rho')$-type.
\end{proof}

\begin{remark}\label{rem:local} Suppose that $\phi \in V_\pi$ is an automorphic form satisfying \eqref{e:minweight} and suppose that $K'$ is a compact open subgroup that acts on $\phi$ by a \emph{character}. Taking $\O = \O^\m$ in Corollary \ref{t:maincor} then gives us the \emph{trivial bound}:
$$\|\phi\|_\infty \ll_{D, \pi_\infty, \eps} [K_{\O^\m}:K']^{1/2}\|\phi\|_2,$$ which is a mild extension of \eqref{e:trivial}.

On the other hand, suppose that $\pi$ has trivial central character and is spherical at all $p|d$. Denote the conductor of $\pi$ by $C$ and let $C_1$ be the smallest integer such that $C|C_1^2$. Suppose that $\phi \in V_\pi$ has the property that some $G(\A_\f)$ translate of it is fixed by the ``principal congruence subgroup" $K_{\O^\m(C_1)}$ (see \eqref{d:congchar} for the definition). Then by the results of \cite[p. 96-97]{Silberger}, the action of $K_\O$ on $\phi$ generates a representation of dimension $\le C_1 \prod_{p|C}(1+p^{-1})$, and so by Corollary \ref{t:maincor}:
\begin{equation}\label{e:localrefined}\|\phi \|_\infty \ll_{D, \pi_\infty} C_1^{1/2} \big(\prod_{p|C}(1+p^{-1})\big)^{1/2}.\end{equation} This is a refinement of the local bound \eqref{e:localgen} for such $\phi$. We note here that the class of $\phi$ having the property mentioned above is quite broad and includes the usual newforms, the automorphic forms of minimal type, and the $p$-adic microlocal lifts, as well as $G(\A_\f)$-translates of all these.\end{remark}

\subsection{The  case of automorphic forms of minimal type}
We now explain how Theorem \ref{t:maingen} implies Theorem B. In fact, we will provide a more general version of Theorem B in Theorem \ref{t:main} below. Before stating the theorem, let us briefly recall the concept of a minimal vector.  Let $p$ be an odd prime and $\pi_p$ be a twist-minimal supercuspidal representation of $\GL_2(\Q_p)$ of conductor $p^{c_p}$. (The twist-minimal condition is automatic whenever $\pi_p$ has trivial central character, or more generally whenever $c_p \neq 2 m_p$ where $p^{m_p}$ is the conductor of the central character of $\pi_p$.) Note that as $\pi_p$ is supercuspidal, we must have $c_p \ge 2$.  We define integers $n_p$, $d_p$ as follows depending on the congruence class of $c_p$ mod 4:

\begin{enumerate}
\item If $c_p \equiv 0 \pmod{4}$, then $n_p = \frac{c_p}2$, $d_p=0$.

\item \label{case2} If $c_p \equiv 2 \pmod{4}$, then $n_p = \frac{c_p}2 - 1$, and $d_p=1.$

\item \label{case3} If $c_p \equiv 1 \pmod{2}$, then $n_p = \frac{c_p}2 + \frac12$, $d_p=0$.

\end{enumerate}

The concept of a minimal vector was first introduced in \cite{HNS} where we focussed only on the first case above, i.e., the case $c_p \equiv 0 \pmod{4}$. In this case, the minimal vector is a unique (up to multiples)\textbf{} vector $\phi_p$ in the space of $\pi_p$ that can be described as follows: Let $\alpha_p \in \Z_p^\times$  be a non-square. Define the order $\tilde{O}_p$ of $\GL_2(\Z_p)$ via $$\tilde{O}_p = \Z_p  + \Z_p \mat{0}{1}{\alpha_p}{0} +  p^{c_p/4}M_2(\Z_p)$$ and let $\O_p = \iota_p^{-1}(\tilde{O}_p)$ be the corresponding order of $D_p$. Then there exists a character $\chi_{\pi_p}$ of $\O_p^\times$ (defined in Definition 2.10 of \cite{HNS}) such that $\pi_p(k_p) \phi_p = \chi_{\pi_p}(k_p) \phi_p,$ for all $k_p \in  \O_p^\times$. This property defines the minimal vector uniquely up to multiples (the definition depends on the isomorphism $\iota_p$ and the element $\alpha_p$ but a different choice of these simply corresponds to a $\O_p^\times$ translate of $\phi_p$).

In a recent work \cite{HN-minimal}, Hu and Nelson extended the concept of a minimal vector to cases \eqref{case2} and \eqref{case3} above, so that now there is a well-defined notion of a minimal vector for all twist-minimal supercuspidal representations of $\GL_2(\Q_p)$ for $p$ odd. We remark here that the twist-minimality is merely for convenience since the minimal vector in the general case is defined in terms of the twist-minimal case. In principle, the case of $p=2$ can also be dealt with similarly but in this case the computations get more technical and these have not been performed so far.
The analogous vectors for the case of principal series representations has also been dealt with in separate work of Nelson \cite{nelson-padic-que}; in this case the relevant vectors are known as $p$-adic microlocal lifts.

Going back to the case of a twist-minimal supercuspidal representation $\pi_p$ of $\GL_2(\Q_p)$ for $p$ odd, we define a ``Type 2 minimal vector" as in \cite{HN-minimal}. If $c_p \not \equiv 2 \pmod{4}$, then the relevant space is one dimensional and so any minimal vector is automatically of Type 2. In the case $c_p \equiv 2 \pmod{4}$, the space of minimal vectors is $p$-dimensional (except for the case $c_p=2$, when the space is $p-1$ dimensional) and has a basis consisting of Type 2 minimal vectors.

A Type 2 minimal vector $\phi_p$ in the space of $\pi_p$ has the property that there exists an order $\O_p \in \O^\m_p$ of level $p^{n_p}$ such that the action of $\O_p^\times$ on $\phi_p$ generates an irreducible representation $\rho_p$ of $\O_p^\times$ with dimension $p^{d_p}$, except in the special case $c_p=2$, in which the representation has dimension $p-1$. Now Theorem \ref{t:maingen} leads to the following theorem.

\begin{theorem}\label{t:main}Let $\pi = \otimes_v \pi_v$ be an irreducible, unitary, cuspidal automorphic
representation of $G(\A)$. Assume that \begin{itemize}
\item If $p|d$, then $\pi_p$ has a (non-zero) vector fixed by $K_p$. (This implies that $\pi_p$ is one-dimensional for each $p|d$.)

\item If $p \nmid d$, and the representation $\pi_p$ of $G_p \cong \GL_2(\Q_p)$ is ramified, then $p$ is odd and $\pi_p$ is a twist-minimal supercuspidal representation with conductor $p^{c_p}$.
\end{itemize} Define $C =\prod_p p^{c_p}$ with the product taken over the ramified primes,  so that $C$ is the (``away from $d$" part of the) conductor of $\pi$ and equal to an odd squarefull integer. Let $\phi=\otimes_v \phi_v$ be a non-zero automorphic form in the space of $\pi$ such that $\phi_p$ is $K_p$ fixed for all $p \nmid C$, $\phi_\infty$ is a vector of smallest non-negative weight, and $\phi_p$ is a Type 2 minimal vector  for each $p|C$. Then we have \begin{equation} \label{e:mintheoprove}
\sup_{g \in G(\A)} |\phi(g) | \ll_{D, \pi_\infty, \eps} C_1^{1/3+\eps} \|\phi\|_2 \prod_{\substack{p|C\\ c_p \equiv 2\bmod{4}}} p^{1/6}. \end{equation}
\end{theorem}
\begin{proof}For each $p|C$, we have a local order $\O_p$ of level $p^{n_p}$ such that the action of $\O_p^\times$ on $\phi_p$ generates a representation $\rho_p$ of dimension $p^{d_p}$, except if $c_p=2$, when the dimension is $p^{d_p}\left(1-\frac1p\right)$. Let $\O$ be the corresponding global order (put $\O_p = \O^\m_p$ if $p \nmid C$) and $\rho = \otimes_p\rho_p$ the corresponding representation of $K_\O$. Then the dimension of $\rho$ is $$r:=\prod_{\substack{p|C\\ c_p \equiv 2\bmod{4}}} p \prod_{c_p=2} (1-\frac1p)$$ and the level of $\O$ is $\prod_{p|C}p^{n_p}$.  Now the result follows immediately from Theorem \ref{t:maingen}.
\end{proof}

Theorem \ref{t:main} improves upon the local bound \eqref{e:localgen} except when $\sqrt{C}$ is a squarefree integer (in which case we recover the local bound).
\subsection{Bounds for $p$-adic microlocal lifts and for newforms}\label{s:otherapps}In fact, Theorem \ref{t:maingen} implies sub-local bounds in the level aspect for certain families of automorphic forms in addition to the ones of minimal type described above. For example, consider the case where $\pi$ has trivial central character and whose (away-from-$d$) conductor $C$ equals $N^4$ where $N=\prod_p p^{n_p}$ is an odd integer.  For two characters $\chi_1$, $\chi_2$ on $\Q_p^\times$, let $\chi_1
\boxplus \chi_2$ denote the principal series representation on
$\GL_2(\Q_p)$
that is unitarily induced from the corresponding representation of its Borel subgroup. Now assume that for each $p|N$, $\pi_p$ is of the form $\chi_p \boxplus \chi_p^{-1}$ with $a(\chi_p)=2n_p$.  Let $\phi_p \in V_{\pi_p}$ at these primes $p$ correspond to \emph{$p$-adic microlocal lifts} in the sense of \cite{nelson-padic-que}.

Consider the group $$K_p^*(n_p) =\left\{\mat{a}{b}{c}{d} \in \GL_2(\Z_p),  b \equiv  c \equiv 0 \bmod p^{n_p}\right\}.$$ Then, by \cite[Lemma 22]{nelson-padic-que}, we see that a $p$-adic microlocal lift $\phi_p \in V_{\pi_p}$ is characterized by the property that $\pi_p(\iota_p^{-1}(k))\phi_p = \chi_p^{\pm1}(ab^{-1})\phi_p$ for all $k = \mat{a}{b}{c}{d} \in   K_p^*(n_p)$. On the other hand, it is easy to check that $\iota_p^{-1}(K_p^*(n_p)) =\O_p^\times$ where  $\O_p$ is an order of level $p^{2n_p}$.  Therefore (by identical arguments as in the proof of Theorem \ref{t:main}), we obtain  \begin{equation}\label{e:microlocal}\|\phi \|_\infty \ll_{D, \pi_\infty, \eps} C^{\frac16+ \eps},\end{equation} which is a sub-local bound  for sup-norms of ``automorphic forms of $p$-adic microlocal type".

 \begin{remark} Minimal vectors and $p$-adic microlocal lifts are analogues of each other, the only difference being that the former live in supercuspidal representations and the latter live in principal series representations. Both these classes of vectors may be viewed as special cases (in the $p$-adic setting) of the more general class of ``localized" vectors. See also \cite{ven-nelson} for a discussion of localized vectors in the archimedean setting, where they are termed ``microlocalized" vectors. \end{remark}

Finally, we discuss what sort of bound Theorem \ref{t:maingen} gives us for \emph{newforms}.  We obtain the following general result:

\begin{theorem}\label{t:newforms}Let $\pi = \otimes_v \pi_v$ be an irreducible, unitary, cuspidal automorphic
representation of $G(\A)$ with conductor $C$. Let $M$ be the conductor of the central character of $\pi$. Let $\phi$ in the space of $\pi$ be a global newform, i.e., $\phi = \otimes_v \phi_v$ with $\phi_p$ spherical if $p \nmid C$, $\phi_p$ equal to the local newvector for $p|C$, and $\phi_\infty$ a vector of smallest non-negative weight. Then we have
\begin{equation}
\sup_{g \in G(\A)} |\phi(g) | \ll_{D, \pi_\infty, \eps} C^{\eps} \min(\max(C^{\frac13}, C_1^{\frac12}),  \  (C_1^2/C)^{-\frac1{24}}\mathrm{lcm}(M, C_1)^{\frac12})\|\phi\|_2. \end{equation}
\end{theorem}
\begin{proof}For any integer $N$, let $\O_0(N)$ denote the Eichler order of level $N$. We first apply Theorem \ref{t:maingen} with $\O =\O_0(C)$. Since $\phi$ transforms by character under the action of $K_{\O_0(C)}$, we obtain that $$\sup_{g \in G(\A)} |\phi(g) | \ll_{D, \pi_\infty, \eps} C^{\eps} \max(C^{1/3}, C_1^{1/2}).$$
Next, we apply the theorem with $\O=\O_0(C')$ where $C' = C_1^2/C$ is the squarefree integer obtained by taking the product of all primes which divide $C$ to an odd power. Then, it was shown in \cite[Sec. 2.7]{saha-sup-level-hybrid} that the action of $K_\O$ on (a suitable right-translate of) $\phi$ generates a representation of dimension $\ll \frac{\mathrm{lcm}(M, C_1)}{C'}$. Now applying Theorem \ref{t:maingen} (and using the fact that right-translating does not change the sup-norm), we get that $$\sup_{g \in G(\A)} |\phi(g) | \ll_{D, \pi_\infty, \eps} C^{\eps}  (C')^{-1/24} \mathrm{lcm}(M, C_1)^{1/2}).$$This completes the proof.
\end{proof}
Theorem \ref{t:newforms} generalizes several previously known bounds for the supnorms of newforms on $G(\A)$, and its proof clearly demonstrates the flexibility of Theorem \ref{t:maingen}.  Note however, that when the central character of $\pi$ is trivial, then Theorem \ref{t:newforms} reduces to \begin{equation}\label{e:newforms}\|\phi \|_\infty \ll_{D, \pi_\infty, \eps} C_1^{\frac12+ \eps} (C')^{-1/24} \end{equation} where $C' = C_1^2/C$ is the squarefree integer obtained by taking the product of all primes which divide $C$ to an odd power. This bound \eqref{e:newforms} \emph{fails} to improve upon the local bound \eqref{e:localgen} for any family of newforms $\phi \in V_\pi$ for which $\frac{C'}{C} \rightarrow 0.$

In particular, a key outstanding case concerns the problem of beating the local bound for newforms of trivial central character in the depth aspect $C=p^n, \ n \rightarrow \infty$ where $p$ is a fixed prime. This case will be treated in a sequel to this paper written with Y. Hu, where we will introduce a new technique in the setting described above (under an additional assumption that $p$ is odd) which will enable us to replace the exponent $1/2$ in \eqref{e:localgen} by the exponent $5/24$ in this particular aspect.

\subsection{Preparations for the proof}
 We now begin the proof of Theorem \ref{t:maingen}. The main part of the proof will be completed in Section \ref{s:amplglobal} (which will crucially rely on the counting results from Section \ref{s:counting}). In this subsection, we make a few simple but key observations, which will allow us to impose additional hypotheses without any loss of generality.

First of all, the property of being  of minimal weight at the archimedean place, strictly speaking, depends on the local isomorphism $\iota_\infty$ which has been fixed by us. However, a different choice of $\iota_\infty$ simply corresponds to replacing $\phi$ by a $G_\infty$-translate of it, and (by definition) the sup-norm of this translated form coincides with the sup-norm of $\phi$. Therefore, fixing $\iota_\infty$ does not change the sup-norm. Now we fix, once and for all, a compact fundamental domain $\J$ for the action of $$\Gamma_{\O^\m} = \{ \gamma \in \iota_\infty(\O^\m), \ \det(\gamma)=1\}$$ on $\H$. Any element of $G(\A)$ can be left-multiplied by a suitable element of $Z(\A)G(\Q)$ and right-multiplied by a suitable element of $K_{\O^\m}$ to get an element $g_\infty \in G_\infty$ such that $\det(\iota_\infty(g_\infty))>0$ and $\iota_\infty(g_\infty)(i)$ lies in $\J$. Since $|\phi(g)|$ is $Z(\A)G(\Q)$ invariant, we may assume, for the purposes of proving \eqref{e:toprove}, that $g = \prod_v g_v$ satisfies \begin{equation}\label{e:toprove2}
g_p \in K_p  \text{ for all } p \in \f, \quad \det(\iota_\infty(g_\infty))>0, \quad \text{ and }  \iota_\infty(g_\infty)(i) \in \J,
\end{equation} where $\J$ is our fixed compact set.

Secondly, suppose that $\O'\subseteq \O^\m$ is any order in the same genus as $\O$.  So there exists $h \in G(\A_\f)$ such that $h\cdot \O = \O'$ (recall the notations from Section \ref{s:orderbasic}). Clearly $\O'$ has the same level as $\O$. Let the finite dimensional  representation $\rho'$ of $K_{\O'}$ be defined via $\rho'(k) = \rho(h^{-1}k h)$ (So $\rho$ and $\rho'$ are isomorphic). Now define the automorphic form $\phi' = \pi(h) \phi$. Then $\phi'$ is of $(K_{\O'}, \rho')$ type and of minimal weight at the archimedean place. Further, it has the same sup-norm as $\phi$, being a translate. So it suffices to prove the Theorem for $\phi'$ (which allows us to change the order from $\O$ to $\O'$). However, a very useful algebraic fact, that we will prove below in Section \ref{s:orderbalanced}, is that each genus of orders contains an order with shape $(M_1, M_2, M_3)$ and level $N$ such that $M_3|N_1$. So, for the purpose of proving Theorem \ref{t:maingen}, we can and will assume the following:
\begin{equation}\label{e:assump1}\text{If }(M_1, M_2, M_3)\text{ is the shape of }\O,\text{ then }M_3\text{ divides }N_1. \end{equation}

Thirdly, we may assume, for the purpose of proving Theorem \ref{t:maingen}, that
\begin{equation}\label{e:assump2}\rho \text{ is irreducible. } \end{equation} Indeed, suppose we have proved the Theorem under \eqref{e:assump2}. Now for the general case, we simply write $\phi$ as an orthogonal sum of automorphic forms $\phi_i$, each of which generates an irreducible representation $\rho_i$ under the action of $K_\O$. Now apply the already proved result to each $\phi_i$, follow it by the triangle inequality and then Cauchy Schwartz, to obtain the desired result for $\phi$ (using the facts that $\dim(\rho) = \sum_i \dim(\rho_i)$ and $\|\phi\|_2^2 = \sum_i \|\phi_i\|_2^2$).

\subsection{A result on balanced representatives for orders}\label{s:orderbalanced}
\begin{definition} Given a pair of lattices $\L_1$, $\L_2$ in $D$ such that $\L_1 \subseteq  \L_2$, \begin{enumerate}
\item The invariant factors of $\L_1$ in $\L_2$ are the unique quadruple of positive integers $(a_1,a_2,a_3,a_4)$ such that $a_1|a_2|a_3|a_4$ and $$ \L_2 /\L_1 \simeq (\Z / a_1\Z) \times (\Z / a_2\Z) \times (\Z / a_3\Z) \times (\Z / a_4\Z).$$

  \item $\L_1$ is balanced in $\L_2$ if the  invariant factors $(a_1,a_2,a_3,a_4)$ have the following property: If $t_1$ denotes the smallest integer such that $a_1a_2a_3a_4$ divides $t_1^2$, then $a_4$ divides  $t_1$.
      \end{enumerate}
\end{definition}
Note that if $\L_1$ and $\L_2$ are orders, then the smallest invariant factor $a_1$ equals 1.
\begin{remark}  Let $\O \subseteq \O^\m$ be an order with shape $(M_1, M_2, M_3)$ and level $N$, and let $N_1$ be the smallest integer such that $N|N_1^2$. Now, suppose that $\O$ is balanced in $\O^\m$. Then by Remark \ref{rem:shapeinv}, we see that $M_3|N_1$. In particular assumption \eqref{e:assump1} holds.
\end{remark}

The object of this subsection is to prove the following result, which was used in the previous subsection to show that we can always assume \eqref{e:assump1} without any loss of generality.
\begin{proposition}\label{p:ordergenusglobal}Let $\O$ be an order in $D$.  Then there exists $g \in G(\A_\f)$ such that $g\cdot \O $ is balanced in $\O^\m$.
\end{proposition}
To prove the above Proposition, we first of all recall (see, e.g., \cite[Chapter 24]{Voight}) that the order $\O$  can be written as $\O= \Z + f \O^{\rm gor}$ where $f \in \Z$ and $\O^{\rm gor}$ is a \emph{Gorenstein order}. If for some $g \in G(\A_\f)$ we know that $g \cdot\O^{\rm gor} \subseteq \O^\m$ and that the invariant factors of  $g \cdot\O^{\rm gor}$ in $\O^\m$ are $(1, a_2, a_3, a_4)$ then it easily follows that the invariant factors of $g \cdot\O$ in $\O^\m$ are $(1, fa_2, fa_3, fa_4)$. In particular, $g \cdot\O$ is balanced in $\O^\m$ whenever $g \cdot \O^{\rm gor}$ is. So it suffices to prove Proposition \ref{p:ordergenusglobal} for Gorenstein orders.

Being Gorenstein is a local property. Now from the local-global principle for orders (see Section \ref{s:orderbasic}), Proposition \ref{p:ordergenusglobal} follows from the next statement.
\begin{proposition}\label{p:ordergenuslocal}Let $\O_p \subseteq \O_p^\m$ be a Gorenstein order of $D_p$. Then there exists an order $O'_p$ of $D_p$ with the following properties:
\begin{enumerate}
\item $\O'_p \simeq \O_p$,
\item $\O'_p \subseteq \O^\m_p$,
\item If $(m_1, m_2, m_3)$ are the unique triple of non-negative integers such that $m_1 \le m_2 \le m_3$ and there is an isomorphism as   $\Z_p$-modules $$\O^\m_p/\O'_p \simeq (\Z_p / p^{m_1}\Z_p) \times (\Z_p / p^{m_2}\Z_p)  \times (\Z_p / p^{m_3}\Z_p), \text{ then} $$ \begin{equation}\label{e:localsmith}m_3 \le \left\lceil \frac{m_1+m_2+m_3}{2} \right\rceil \text{ holds.}\end{equation}.
\end{enumerate}
\end{proposition}
We now prove Proposition \ref{p:ordergenuslocal}. We rely heavily on the work of Brzezinski \cite{brzezinski} who gives a complete list of Gorenstein orders (up to isomorphism) and their resolutions in terms of explicit linear combinations of generators of $\O^\m$. It is therefore easy (albeit tedious) to compute the triple $(m_1, m_2, m_3)$ for each order in his list (by bringing the corresponding matrices to Smith normal form). We do this and observe that most orders in his list already satisfy \eqref{e:localsmith}; the ones that aren't can be conjugated by a simple element and made to satisfy it. We give the key details below, omitting some of the routine calculations.

First consider the case when $D_p \simeq M_2(\Q_p)$. Put $x_1 = \mat{1}{0}{0}{0}$, $x_2 = \mat{0}{1}{0}{0}$, $x_3 = \mat{0}{0}{1}{0}$. Note that $\O_p^\m = \langle 1, x_1, x_2, x_3 \rangle$. According to Prop. 5.4 of \cite{brzezinski}, $\O_p$ is isomorphic to one of the cases $(a) - (d_3')$ described there. We denote $r_n = \mat{p^{\lfloor \frac{n}{2} \rfloor}}{0}{0}{1}$. We write down the required order $\O_p'$ in each case,  using the notation from Proposition 5.4 of \cite{brzezinski}.

Case (a). In this case we take $\O_p' = r_n E_n^{(1)}r_n^{-1}$.

Case (b). In this case we take $\O_p' = E_n^{(-1)}$.

Case (c). In this case we take $\O_p' = E_n^{(0)}$.

Case ($d_1$).  In this case we take $\O_p' = r_n E_{n,s}^{(1)}r_n^{-1}$.

Case ($d_2$).  In this case we take $\O_p' =  E_{n,s}^{(-1)}$.

Case ($d_3$).  In this case we take $\O_p' =  E_{n,s}^{(0)}$.

Case ($d_3'$).  In this case we take $\O_p' =  E_{2,s^+}^{(0)}$.

Next, we consider the case when $p|d$, i.e, $D_p$ is a division algebra. Let $x_1$, $x_2$, $x_3$ be as in \cite[(5.5)]{brzezinski}. Then $\O_p^\m = \langle 1, x_1, x_2, x_3 \rangle$. According to Prop. 5.6 of \cite{brzezinski}, $\O_p$ is isomorphic to one of the cases $(a) - (c_2)$ described there.

Case (a). In this case we take $\O_p' = \Gamma_n^{(-1)}$.

Case (b). In this case we take $\O_p' = \Gamma_n^{(0)}$.

Case ($c_1$).  In this case we take $\O_p' =\Gamma_{n,s}^{(-1)}$.

Case ($c_2$).  In this case we take $\O_p' =  \Gamma_{n,s}^{(0)}$.

In all cases above, the description of $\O_p'$ given in \cite{brzezinski} provides an explicit $\Z_p$-basis for $\O_p'$ in terms of a $\Z_p$-basis for $\O^\m_p.$ We reduce the resulting matrix into Smith normal form via elementary operations, and observe that the invariant factors $(1, p^{m_1}, p^{m_2}, p^{m_3})$ of the resulting matrix always satisfies \eqref{e:localsmith}. (We remark here that in the case $(d_2)$ for $p \nmid d$, and the cases (b), $(c_1)$ for $p | d$, the reference \cite{brzezinski} gives five generators for $\O_p'$. However, once the generator matrix is brought into Smith normal form, we get exactly four non-zero rows; these correspond to a $\Z_p$-basis of the form required.)

This completes the proof of Proposition \ref{p:ordergenuslocal}, and therefore of Proposition \ref{p:ordergenusglobal}.

\section{Amplification}\label{s:amplglobal}
In this Section, we complete the proof of Theorem \ref{t:maingen}. Throughout this section, we assume the setup of Section \ref{s:mainresultsubsec}. We fix an order $O\subseteq \O^\m$ of level $N$, an automorphic form $\phi$ in $V_\pi$ with $$\|\phi\|_2 =1,$$ and a finite dimensional representation $\rho$ of $K_\O$ such that the conditions of Theorem \ref{t:maingen} are satisfied. Given the above data, and some $g \in G(\A)$ satisfying \eqref{e:toprove2}, our goal in this section is to prove
\begin{equation}\label{e:toprovefinal}
|\phi(g) | \ll_{D, \pi_\infty, \eps} N^{\eps} \min(\max(N^{1/3}, N_1^{1/2}), N^{11/24}) \dim(\rho)^{1/2},
\end{equation}
which will complete the proof of Theorem \ref{t:maingen}.  As explained previously, we can and will assume (without loss of generality) that \eqref{e:assump1} and \eqref{e:assump2} hold.
\subsection{Test functions}
 Our main tool is the amplification method. From the adelic point of view, amplification corresponds to an appropriate choice of test function $\kappa$ on $G(\A)$ which increases the contribution of the particular automorphic form $\phi$ in the resulting  pre-trace formula. In this subsection, we describe this test function $\kappa$ (which will depend on $\phi$ and $\O$) and note its key properties.

 Recall that $\f$ denotes the set of finite primes, which we identify with the non-archimedean places of $\Q$. The representation $\rho$ is isomorphic to $\otimes_{p \in \f} \rho_p$ where $\rho_p$ is a representation of $\O_p^\times$ with $\rho_p$ trivial for almost all $p$. We choose a finite subset $S\subset \f$ with the following properties:
 \begin{enumerate}
 \item $S$ contains all primes dividing $dN$,

 \item If $p \notin S$, then $\rho_p$ is trivial.

 \end{enumerate}
For convenience, we denote $G_S = \prod_{p \in S}G_p$, $\Q_S^\times = \prod_{p \in S} \Q_p^\times$, and $\O_S^\times = \prod_{p\in S}\O_p^\times$. Put $\rho_S = \otimes_{p \in S}\rho_p$. So $\rho_S$ is an irreducible representation of $\O_S^\times$ with $\dim(\rho_S) = \dim(\rho)$ and $V_{\rho_S}$ is the subspace of $V_\pi$ generated by the action of $\O_S^\times$ on $\phi$. It follows that $\rho_S$ is unitary with respect to the Petersson inner product.
\begin{comment}For convenience, given an algebraic group $H$ defined over $\Q$ (we will take $H$ to be one of the groups $G$, $\Z$, $G'$) and given a subset $T$ of $\f$ (resp., a positive integer $M$), we adopt the notation $H_T = \prod_{p\in T}H(\Q_p)$ (resp., $H_M = \prod_{p |M}H(\Q_p)$.\end{comment}
Let $\ur = \f \setminus S$ be the set of primes not in $S$. We will choose $\kappa$ of the form $\kappa = \kappa_S \kappa_\ur \  \kappa_\infty$.

We define the function $\kappa_S$ on $G_S$ as follows:
$$\kappa_S (g_S)=  \begin{cases}  0 &\text{ if } g_S \notin  \Q_S^\times \O_S^\times, \\ \omega_{\pi}^{-1}(z) \langle \phi, \pi(k) \phi \rangle &\text{ if }  g_S=  zk, \quad z \in \Q_S^\times, \ k \in  \O_S^\times. \end{cases}$$

Our assumptions and basic properties of finite dimensional irreducible representations of compact groups imply that  \begin{equation}\label{kappaN}\begin{split}R(\kappa_S)\phi &:= \int_{\Q_S^\times \bs G_S} \kappa_S(g)(R(g)\phi)  \ dg \\&=  \int_{\O_S^\times} \langle \phi, \rho_S(k) \phi \rangle (\rho_S(k) \phi) \ dk \\&= \frac{\vol(K_\O)}{\dim(\rho)}\phi, \end{split}\end{equation}
where we have used the fact that $\vol(\O_S^\times) = \vol(K_\O)$ (since $\O_p^\times = (\O^\m_p)^\times$ for primes $p \notin S$).   Observe also that the formula in the first line above gives us a self-adjoint, non-negative operator $R(\kappa_S)$  on the space of all automorphic forms on $D^\times(\A)$ which have central character $\omega_\pi$.

Next, we consider the primes $p \in \ur$. Note that $\pi_p$ is unramified for each such prime (indeed for such $p$, $\rho_p$ is trivial and hence $\phi$ is $K_p$-fixed). Let $\mathcal{H}_{\ur}$ be the set of all  compactly supported functions on $\prod_{p \in \ur}\GL_2(\Q_p)$ that are bi-$\GL_2(\Z_p)$ invariant for each $p \in \ur$ and transform under the action of the centre by $\omega_{\pi}^{-1}$. For each positive integer $\ell$ satisfying $(\ell, S)=1$,  define the functions $\kappa_\ell$ in $\mathcal{H}_{\ur}$ as in Section 3.5 of \cite{saha-sup-level-hybrid}; these correspond to the usual Hecke operators $T_\ell$.

Recall that for each $p \in \ur$, we have fixed an isomorphism $\iota_p: D_p  \xrightarrow {\cong} M(2,\Q_p)$ such that $\iota_p(\O_p) = M(2,\Z_p)$. Put $G_\ur:=\prod_{p\in \ur}G_p$, $K_\ur = \prod_{p\in \ur}\O_p^\times$. Using the local isomorphisms $\iota_p$, we identify $\mathcal{H}_{\ur}$ with the set of  compactly supported functions on $G_\ur$ that are bi-$K_\ur$ invariant and transform under the action of the centre by $\omega_{\pi}^{-1}$. (This identification  does not depend on the choice of the local isomorphisms.)  In particular, we can now identify the  functions $\kappa_\ell$ with functions on $G_\ur$.  For each $\kappa \in \mathcal{H}_{\ur}$, we obtain in the usual manner an operator $R(\kappa)$ on the space of all automorphic forms on $D^\times(\A)$ which have central character $\omega_\pi$.  We have the standard involution $\kappa \mapsto \kappa^*$ on $\mathcal{H}_{\ur}$ given by $\kappa^*(g) = \overline{\kappa(g^{-1})}$ which makes $R(\kappa^*)$  the adjoint of $R(\kappa)$. Given elements $\kappa_1$ and $\kappa_2$ in $\mathcal{H}_{\ur}$, we define their convolution $\kappa_1 \ast \kappa_2 \in \mathcal{H}_{\ur}$ to be the function defined as follows: \begin{equation}(\kappa_1 \ast \kappa_2) (h) = \int_{Z\bs G} \kappa_1(g^{-1})\kappa_2(gh) dg. \end{equation} Note that   $R(\kappa_1 \ast \kappa_2) = R(\kappa_1)R(\kappa_2)$.

For each positive integer $m$ such that $(m,S)=1$, we let $\lambda_\pi(m)$ be the coefficient of $m^{-s}$ in the Dirichlet series corresponding to $L(s, \pi)$, where we normalize the $L$-function to have functional equation $s \mapsto 1-s$.

Let $\Lambda \ge 1$ be a real number. We define $$\P= \{\ell: \ell \text{ prime, } \ell \in \ur, \ \Lambda \le \ell \le 2\Lambda\}. $$
Define for each integer $r$ satisfying $(r,S)=1$, $c_r = \frac{|\lambda_\pi(r)|}{\lambda_\pi(r)}.$
We put $\delta_\ur = \sum_{r \in \P} c_r \kappa_r$, and  $\gamma_\ur = \sum_{r \in \P} c_{r^2} \kappa_{r^2}$. Finally, put  $$\kappa_\ur = \delta_\ur \ast \delta_\ur^* + \gamma_\ur \ast \gamma_\ur^*.$$ It is clear that $R(\kappa_\ur)$ is a normal, non-negative operator. Moreover,  by a standard argument (see (5.6-5.8) of \cite{blomer-harcos-milicevic} and Section 3.7 of \cite{saha-sup-level-hybrid}) we get that \begin{equation}\label{kappaur}R(\kappa_\ur )\phi = \lambda_\ur \phi, \quad \lambda_\ur \gg_\eps \Lambda^{2-\eps}.\end{equation} Furthermore, the well-known relation $$\kappa_m \ast \kappa_n^\ast = \sum_{t|
\gcd(m,n)} \left(\prod_{p|t}\omega_{\pi_p}(t)\right) \left(\prod_{p|n}\omega_{\pi_p}^{-1}(n)\right) \kappa_{mn/t^2},$$
gives us that \begin{equation}\label{e:ylkappa}\kappa_\ur = \sum_{1 \le l \le 16\Lambda^4} y_l \kappa_l\end{equation} where the complex numbers $y_l$ satisfy:

\begin{equation}\label{e:ylcases}|y_l| \ll \begin{cases}\Lambda, & l=1,\\1, &l =  \ell_1 \ell_2 \text{ or } l = \ell_1^2\ell_2^2 \text{ with } \ell_1, \ell_2 \in \P, \\ 0, &\text{otherwise.}\end{cases}\end{equation}

Finally, we consider the infinite place. As we are not looking for a
bound in the archimedean aspect, the choice of $\kappa_\infty$ is
unimportant. However for definiteness, let us fix the  function
$\kappa_\infty$ on $G_\infty$ as in \cite{saha-sup-level-hybrid} (see end of Section 3.5). In particular this choice has the property that $$\kappa_\infty(g) \neq 0 \Rightarrow
\det(\iota_\infty(g_\infty))>0, \ u(\iota_\infty(g_\infty)) \le 1$$ where for $g \in \GL_2(\R)^+$, we let $u(g) = \frac{|g(i) - i|^2}{4 \Im(g(i))}$ denote the hyperbolic distance from $g(i)$ to $i$.
Furthermore, the operator $R(\kappa_\infty)$ is self-adjoint, non-negative  and satisfies
\begin{equation}\label{kappainf}R(\kappa_\infty)\phi_\infty =
\lambda_\infty \phi_\infty, \quad \lambda_\infty \gg_{\pi_\infty} 1.\end{equation}

\subsection{The automorphic kernel and spectral expansion}Note that the function $\kappa$ on $G(\A)$ defined in the previous subsection transforms by $\omega_\pi^{-1}$ under the action of the centre $\A^\times$. In particular, $\kappa$ is $\Q^\times$-invariant. Define the automorphic kernel $K_\kappa(g_1, g_2)$ for $g_1, g_2 \in G(\A)$
via $$K_\kappa(g_1, g_2) = \sum_{\gamma \in  G'(\Q)} \kappa(g_1^{-1} \gamma
g_2).$$

Using \eqref{e:volumeorder}, \eqref{kappaN}, \eqref{kappaur}, and \eqref{kappainf}, we see that $$R(\kappa)\phi = \lambda \phi, \quad \text{where } \lambda \gg_{\pi_\infty,  \eps} \frac{N^{-1-\eps} \Lambda^{2-\eps}}{\dim(\rho)}.$$

Now, spectrally expanding $K_\kappa( g  , g)$  and using the non-negativity of the operator $R(\kappa)$, we obtain,

\begin{equation}\label{keyineq1}\dim(\rho)^{-1} N^{-1-\eps}\Lambda^{2-\eps} |\phi(g)|^2 \ll_{\pi_\infty,  \eps}    K_\kappa(g,g).\end{equation}

Write $K_S = \prod_{p \in S}K_p$.  Now, let $g \in G(\A)$ satisfy \eqref{e:toprove2}; so $g_S \in K_S$. Put $z = g_\infty(i) \in \J$. Using \eqref{e:ylkappa}, we see that \begin{equation}\label{keyineqha}K_\kappa(g,g) =  \sum_{1 \le l \le 16\Lambda^4} y_\ell\sum_{\gamma \in G'(\Q)} \kappa_S(g_S^{-1}\gamma_S g_S) \kappa_\ell(\gamma) \kappa_\infty(g_\infty^{-1}\gamma_\infty g_\infty).\end{equation}

\begin{definition}\label{d:four}Given a lattice $\L$, $z \in \H$, and a positive integer $\ell$, let  $S(\ell, \L;  z)$ be the set of $\gamma \in G'(\Q)$ satisfying the following properties:

\begin{enumerate}

\item  $\gamma_p \in \Q_p^\times \L_p(\ell)$ for all $p \in \f$, where $\L_p(\ell) = \{ \alpha \in \L_p: \nr(\alpha) \in \ell \Z_p^\times\}$.

\item $\det(\iota_\infty(\gamma_\infty))>0$, $u(z, \iota_\infty(\gamma_\infty) z) \le 1$.
\end{enumerate}
\end{definition}

 Our definition of $\kappa$  implies that if $\kappa_S(g_S^{-1}\gamma_S g_S) \kappa_\ell(\gamma) \kappa_\infty(g_\infty^{-1}\gamma_\infty g_\infty) \neq 0$ then we must have $\gamma \in S(\ell, \ g\cdot\O; \ z)$.  Since $\kappa_\ell(\gamma) \le \ell^{-1/2}$, and $|\kappa_S| \le 1$, $|\kappa_\infty| \le 1$, the triangle inequality on \eqref{keyineqha}, together with \eqref{e:ylcases} and \eqref{keyineq1}, now gives us

 \begin{equation}\label{keyineq2}\begin{split}
|\phi(g)|^2 &\ll_{\pi_\infty, \eps} \dim(\rho) N^{1+\eps} \Lambda^{-2+\eps} \sum_{1 \le \ell \le 16\Lambda^4} \frac{|y_\ell|}{\sqrt{\ell}} \ \big\vert S(\ell, \ g \cdot \O;  z)\big\vert \\ &\ll \dim(\rho) N^{1+\eps} \Lambda^{\eps}\bigg(\Lambda^{-1}  \ \big\vert S(1, \ g \cdot \O;  z)\big\vert + \Lambda^{-3} \sum_{\ell_1, \ell_2 \in \P} \ \big\vert S(\ell_1 \ell_2, \ g \cdot \O;  z)\big\vert  \\ & \quad \quad \quad \quad \quad \quad \quad \quad +     \Lambda^{-4} \sum_{\ell_1, \ell_2 \in \P} \ \big\vert S(\ell_1^2 \ell_2^2, \ g \cdot \O;  z)\big\vert \bigg).\end{split}
\end{equation}
\subsection{The endgame}
We can now wrap up the proof, beginning with a simple proposition that links it all back to Section \ref{s:counting}.

\begin{proposition} For any lattice $\L$ and non-zero integer $\ell$, the natural map $G(\Q) \rightarrow G'(\Q)$ induces a bijection of finite sets $\pm1 \bs \L(\ell;z,1) \cong S(\ell, \L; z).$ In particular $$|S(\ell, \L; z)| = \frac12 | \L(\ell;z,1)|.$$
\end{proposition}
\begin{proof}
It is clear that any element of  $\L(\ell;z,1)$ satisfies the two conditions defining $S(\ell, \L; z)$. Furthermore, if two elements $\gamma_1$, $\gamma_2$ in $\L(\ell;z,1)$ represent the same class in $G'(\Q)$, then putting $t \gamma_1 = \gamma_2$, we obtain (taking norms) that $t^2=1$ which means that $\gamma_1 = \pm \gamma_2$. Therefore we get an injective map $\pm1 \bs \L(\ell;z,1) \rightarrow S(\ell, \L; z)$. To complete the proof, we need to show that this map is surjective. Let $\gamma \in G(\Q)$ be an element whose image in $G'(\Q)$ lies in $S(\ell, \L; z)$. We need to prove that there exists $t_0 \in \Q^\times$ such that $t_0 \gamma \in \L(\ell;z,1)$. By Definition \ref{d:four}, we can find for each prime $p$, an element $t_p \in \Q_p^\times$ such that $\gamma_p \in t_p \L_p(\ell)$ for all primes $p$. By strong approximation for $\Q$, we can choose $t_0\in \Q$ such that $t_0 t_p \in \Z_p^\times$ for all primes $p$. Now consider the element $t_0 \gamma$. For each prime $p$, we have that the $p$-component of $t_0\gamma$ lies in $\L_p$. So by \eqref{e:localization}, we have $t_0 \gamma \in \L$. Furthermore, we have that the $p$-component of $\nr(t_0 \gamma)$ lies in $\ell \Z_p^\times$ for all primes $p$, and hence $\nr(t_0 \gamma) = \pm \ell$. But by assumption $\nr(\gamma)>0$. Hence $\nr(t_0 \gamma) = \ell$. It follows that $t_0 \gamma \in \L(\ell)$. Since $u(z, \iota_\infty(\gamma) z) \le 1$, it is now immediate that $t_0 \gamma \in \L(\ell;z,1)$.
\end{proof}

Now let us go back to \eqref{keyineq2}. We will prove two bounds. For the first, we choose $\Lambda = \frac12 C^{1/4} N^{1/12}$ and apply Proposition \ref{t:countingnew} to \eqref{keyineq2}. (Note here that $g \cdot \O$ has the same level $N$ as $\O$). This gives us \begin{equation}\label{e:final1}|\phi(g)| \ll_{\pi_\infty, \eps} \dim(\rho)^{1/2} N^{\frac{11}{24} + \eps}.\end{equation}

For the second bound, we apply Proposition \ref{t:counting} to \eqref{keyineq2}.
By the assumption \eqref{e:toprove2}, the orders $\O$ and $g\cdot \O$ have the same shape, which we denote by $(M_1, M_2, M_3)$. Furthermore, by \eqref{e:assump1}, $M_1M_2 \gg N/N_1$. Now, applying Proposition \ref{t:counting} to \eqref{keyineq2}, we get $$|\phi(g)|^2 \ll_{\pi_\infty, \eps} \dim(\rho) \  N^{1 + \eps} \left(\Lambda^{-1} + \frac{\Lambda^2}{N} + \frac{N_1}{N} \right).$$ Choosing $\Lambda = N^{1/3}$ above gives us
\begin{equation}\label{e:final2}|\phi(g)| \ll_{\pi_\infty, \eps} \dim(\rho)^{1/2} N^\eps\max(N^{1/3}, N_1^{1/2}).\end{equation}
Combining \eqref{e:final1} and \eqref{e:final2}, we obtain \eqref{e:toprovefinal}.

 \bibliography{sup-compact}

% -----------------
\end{document}